\pdfoutput=1
\documentclass{amsart}
\usepackage[usenames,dvipsnames]{xcolor}
\usepackage{hyperref}
\usepackage{microtype}
\usepackage{adjustbox}
\definecolor{our-color}{rgb}{0.15,0.15,0.6}
\hypersetup{colorlinks=true,linkcolor=our-color,citecolor=our-color,citecolor=our-color,urlcolor=our-color}
\usepackage{chngcntr}
\usepackage{mathtools}
\usepackage{enumitem}
\usepackage[alphabetic]{amsrefs}
\usepackage{amsmath,amssymb,stmaryrd}
\usepackage{tikz,bbm}
\usepackage{bbold}
\usetikzlibrary{matrix,arrows,decorations.pathmorphing, cd}
\usepackage[mathscr]{euscript}
\usepackage{comment}
\usepackage{todonotes}

\swapnumbers

\numberwithin{equation}{section}
\usepackage[nameinlink,capitalise,noabbrev]{cleveref}
\counterwithout{subsubsection}{subsection}
\makeatletter
\let\c@equation\c@subsection

\let\c@subsubsection\c@subsection

\makeatother

\DeclareMathOperator{\Der}{D}
\DeclareMathOperator{\id}{id}
\DeclareMathOperator{\opname}{op}
\newcommand{\Stovicek}{\v{S}\v{t}ov\'{i}\v{c}ek}
\newcommand{\KQ}{\mathrm{KQ}}
\newcommand{\overbar}[1]{\mkern 1.5mu\overline{\mkern-1.5mu#1\mkern-1.5mu}\mkern 1.5mu}
\newcommand{\op}{^{\opname}}
\newcommand{\adjto}{\rightleftarrows}
\newcommand{\Mid}		{\,\big|\,}
\newcommand{\SET}[2]	{\big\{\,#1\Mid#2\,\big\}}

\newcommand{\Hom}       {\operatorname{Hom}}

\newcommand{\ihom}      {\mathsf{hom}}

\newcommand{\spec}      {\operatorname{Spec}}
\newcommand{\Spc}       {\operatorname{Spc}}
\newcommand{\supp}      {\mathrm{supp}}
\newcommand{\Supp}      {\mathrm{Supp}}

\newcommand{\Fun}       {\operatorname{Fun}}

\newcommand{\CAlg}      {\operatorname{CAlg}}
\newcommand{\Tot}       {\operatorname{Tot}}

\newcommand{\Ho}        {\operatorname{Ho}}

\newcommand{\Pos}		{\mathsf{Pos}}
\newcommand{\Lat}       {\mathsf{Lat}}
\newcommand{\DLat}      {\mathsf{DLat}}
\newcommand{\CFr}       {\mathsf{CFrm}}
\newcommand{\Frm}       {\mathsf{Frm}}

\newcommand{\loc}       {\mathrm{loc}}
\newcommand{\perf}      {\mathrm{perf}}

\newcommand{\idem}      {\mathrm{idem}}

\newcommand{\Sp}        {{\mathrm{Sp}}}
\renewcommand{\Pr}      {{\mathsf{Pr}^{L}_{\mathrm{st}}}}
\newcommand{\Cat}       {{\mathsf{Cat}}}
\newcommand{\Poset}     {{\mathsf{Poset}}}
\newcommand{\Set}       {{\mathsf{Set}}}
\newcommand{\Top}       {{\mathsf{Top}}}
\newcommand{\Spec}		{{\mathsf{Spec}}}
\newcommand{\Ab}		{{\mathsf{Ab}}}

\newcommand{\unit}{\mathbbm{1}}
\newcommand{\unitS}{\unit_{\cat S}}
\newcommand{\unitT}{\unit_{\cat T}}

\newcommand{\cat}{\mathscr}

\newcommand{\im}{\operatorname{im}}

\newcommand{\CP}        {{\mathcal{P}}}
\newcommand{\CQ}        {{\mathcal{Q}}}

\newcommand{\CU}        {{\mathcal{U}}}

\newcommand{\Z}         {{\mathbb{Z}}}

\newcommand{\C}         {{\mathscr{C}}}

\newcommand{\D}         {{\mathscr{D}}}

\newcommand{\U}         {{\mathscr{U}}}

\newcommand{\Loc} 			{\operatorname{Loc}}

\newcommand{\colim}  {\operatornamewithlimits{\underset{\longrightarrow}{lim}}}

\renewcommand{\mod}[1]{\mathrm{Mod}_{#1}}

\newcommand{\thick}[1]{\mathrm{thick}\langle{#1}\rangle}
\newcommand{\thickid}[1]{\mathrm{thickid}\langle {#1}\rangle}

\renewcommand{\loc}[1]{\mathrm{loc}\langle{#1}\rangle}
\newcommand{\locid}[1]{\mathrm{locid}\langle {#1}\rangle}

\DeclareMathOperator{\Locideal}{Locid}
\DeclareMathOperator{\Thickideal}{Thickid}
\DeclareMathOperator{\Smashideal}{Smashid}
\DeclareMathOperator{\RadThickideal}{RadThickid}
\newcommand{\Locid}{\Locideal}
\newcommand{\Thickid}{\Thickideal}
\newcommand{\Smashid}{\Smashideal}
\newcommand{\RadThickid}{\RadThickideal}

\newcommand{\ZarT}{\RadThickid(\cat T)}
\newcommand{\ZarTf}{\ZarT^f}
\newcommand{\ZarTY}{\RadThickid(\cat T;Y)}
\newcommand{\ZarTYf}{\ZarTY^f}

\newcommand{\Cl}{\operatorname{Cl}}

\newcommand\noloc{%
  \nobreak
  \mspace{6mu plus 1mu}
  {:}
  \nonscript\mkern-\thinmuskip
  \mathpunct{}
  \mspace{2mu}
}

\newtheorem{theorem}[equation]{Theorem}

\newtheorem{lemma}[equation]{Lemma}
\newtheorem{proposition}[equation]{Proposition}
\newtheorem{corollary}[equation]{Corollary}
\theoremstyle{definition}
\newtheorem{remark}[equation]{Remark}
\newtheorem{definition}[equation]{Definition}
\newtheorem{example}[equation]{Example}
\newtheorem{recollection}[equation]{Recollection}
\newtheorem{construction}[equation]{Construction}

\newtheorem{notation}[equation]{Notation}

\swapnumbers
\theoremstyle{plain}
\newtheorem{thmx}{Theorem}

\author[Barthel]{Tobias Barthel}
\author[Castellana]{Nat{\`a}lia Castellana}
\author[Heard]{Drew Heard}
\author[Naumann]{Niko Naumann}
\author[Pol]{Luca Pol}
\author[Sanders]{Beren Sanders}
\date{\today}

\address{Tobias Barthel, Max Planck Institute for Mathematics, Vivatsgasse 7, 53111 Bonn, Germany}
\email{tbarthel@mpim-bonn.mpg.de}
\urladdr{https://sites.google.com/view/tobiasbarthel/home}

\address{Nat{\`a}lia Castellana, Departament de Matem\`atiques, Universitat Aut\`onoma de Barcelona, 08193 Bellaterra, Spain}
\email{natalia@mat.uab.cat}
\urladdr{http://mat.uab.cat/$\sim$natalia}

\address{Drew Heard, Department of Mathematical Sciences, Norwegian University of Science and Technology, Trondheim}
\email{drew.k.heard@ntnu.no}
\urladdr{https://folk.ntnu.no/drewkh/}

\address{Niko Naumann, Fakult\"{a}t f\"{u}r Mathematik, Universit\"{a}t Regensburg, 93040 Regensburg, Germany}
\email{niko.naumann@mathematik.uni-regensburg.de}
\urladdr{https://homepages.uni-regensburg.de/~nan25776/}

\address{Luca Pol, Fakult\"{a}t f\"{u}r Mathematik, Universit\"{a}t Regensburg, 93040 Regensburg, Germany}
\email{luca.pol@ur.de}
\urladdr{https://sites.google.com/view/lucapol/}

\address{Beren Sanders, Mathematics Department, UC Santa Cruz, 95064 CA, USA}
\email{beren@ucsc.edu}
\urladdr{http://people.ucsc.edu/$\sim$beren/}

\title{Descent in tensor triangular geometry}

\begin{document}

\begin{abstract}
	We investigate to what extent we can descend the classification of localizing, smashing and thick ideals in a presentably symmetric monoidal stable $\infty$-category $\C$ along a descendable commutative algebra $A$. We establish equalizer diagrams relating the lattices of localizing and smashing ideals of~$\C$ to those of $\mod{A}(\C)$ and $\mod{A\otimes A}(\C)$. If $A$ is compact, we obtain a similar equalizer for the lattices of thick ideals which, via Stone duality, yields a \mbox{coequalizer} diagram of Balmer spectra in the category of spectral spaces. We then give conditions under which the telescope conjecture and stratification descend from $\mod{A}(\C)$ to $\C$.  The utility of these results is demonstrated in the case of faithful Galois extensions in tensor triangular geometry.
\end{abstract}



\thanks{TB is supported by the European Research Council (ERC) under Horizon Europe (grant No.~101042990). NC is partially supported by Spanish State Research Agency project PID2020-116481GB-I00, the Severo Ochoa and María de Maeztu Program for Centers and Units of Excellence in R$\&$D (CEX2020-001084-M), and the CERCA Programme/Generalitat de Catalunya. DH is supported by grant number TMS2020TMT02 from the Trond Mohn Foundation. NN and LP are supported by the SFB 1085 Higher Invariants in Regensburg. BS is supported by NSF grant~DMS-1903429. Finally, we would like to thank the Hausdorff Research Institute for Mathematics for their hospitality in the context of the Trimester program \emph{Spectral Methods in Algebra, Geometry, and Topology} funded by the Deutsche Forschungsgemeinschaft under Germany’s Excellence Strategy – EXC-2047/1 – 390685813.}

\maketitle
{\hypersetup{linkcolor=black}
	\vspace*{-2.5\baselineskip}
\tableofcontents
\vspace*{-2.5\baselineskip}
}
\newpage

\section{Introduction}
Let $\C$ be a tensor triangulated (``tt'') category, i.e., a triangulated category equipped with a compatible symmetric monoidal structure. There are a number of ways we can study the tensor triangular geometry of $\C$; for example, we could ask for a classification of the thick ideals, a classification of the localizing ideals, or a classification of the smashing ideals of $\C$. One approach to achieving this is to find a tensor triangulated functor $F\colon \C \to \D$ to a simpler tensor triangulated category~$\D$ and reduce the classification problems in $\C$ to those in $\D$ via descent methods.

One way to produce such functors is to work at the level of underlying $\infty$-categories (i.e., assume $\C$ is a \emph{tt-$\infty$-category}) and consider the base change functor \[F\colon\C \to \mod{A}(\C)\] associated to a highly structured commutative algebra $A$ in $\C$. In fact, under mild assumptions on $\cat C$ and $\cat D$, any functor $F \colon \cat C \to \cat D$ whose right adjoint is conservative is, up to equivalence, base change along a commutative algebra object in $\cat C$ (see \Cref{prop:basechange}). 

An important first step in understanding descent in tt-geometry was taken by Balmer \cite{Balmer2016}, who studied extensions along \emph{separable} commutative algebras.\footnote{This does not require the $\infty$-categorical machinery due to the special nature of separable algebras; see \Cref{comparison-with-Balmer} and \cite{Balmer2011}.} These are commutative algebras~$A$ in $\cat C$ whose multiplication $\mu \colon A \otimes A \to A$ has a bimodule section. In order to state Balmer's theorem, recall that the (radical) thick ideals of an essentially small tt-category $\cat K$ are classified by a spectral topological space $\Spc(\cat K)$ called the Balmer spectrum. Then, given a separable commutative algebra of finite degree $A$ in $\cat K$, Balmer shows that there is a coequalizer diagram of topological spaces \begin{equation}\label{eq:balmer-coequalizer} \begin{tikzcd}[column sep=large] \Spc(\mod{A\otimes A}(\cat K))\arrow[r, shift left] \arrow[r, shift right]& \Spc(\mod{A}(\cat K)) \arrow[r, ""] & \supp(A). \end{tikzcd} \end{equation} In particular, if $A$ is \emph{descendable}\footnote{Balmer uses the terminology \emph{nil-faithful}.} (\cref{def:descendable}) then $\supp(A) = \Spc(\cat K)$ and this coequalizer computes the spectrum of $\cat K$ itself.

In general, the assumptions in Balmer's theorem that the algebra $A$ is separable and in some sense ``small'' (reflected by the assumption that $A \in \cat K$) cannot be removed. For example, the case where $\cat K$ is the category of $K(2)$-local dualizable spectra was studied in detail by a subset of the authors in \cite{BHN}. Morava $E$-theory provides a descendable algebra $A$ in the bigger category of all $K(2)$-local spectra which is not dualizable (i.e., is not contained in $\cat K$) and \cite[Proposition 5.11]{BHN} shows that the analog of the above coequalizer is not a coequalizer of topological spaces, but rather of \emph{spectral spaces}. As that example amply demonstrates, coequalizers, and more generally colimits,  in spectral spaces can behave quite differently than colimits in topological spaces; see also \Cref{ex:spectral-vs-top-colimit}.

It is a classical result, essentially due to Stone, that the category of spectral spaces is anti-equivalent to the category of distributive lattices. The above results, along with previous work on the Balmer spectrum (for example \cite{Kock-Pitsch}), suggest a lattice-theoretic approach to the problem of descent in tt-geometry. This is the approach we take in this paper. 

Throughout, we work in the context of a rigidly-compactly generated tt-$\infty$-category~$\cat C$; that is, a presentably symmetric monoidal stable $\infty$-category $\cat C$ whose homotopy category is rigidly-compactly generated; see the precise definitions in \cref{sec:categories}. The full subcategory of compact objects is denoted $\cat C^c$ (which, in examples, corresponds to the $\cat K$ above). For example $\cat C$ could be the derived $\infty$-category $\Der(R)$ of unbounded chain complexes of $R$-modules for a commutative ring $R$. We study descent of the following posets:
	\begin{enumerate}
		\item The poset $\Locid(\cat C)$ of localizing ideals of $\cat C$. 
		\item The poset $\Smashid(\cat C)$ of smashing localizing ideals of $\cat C$. 
		\item The poset $\Thickid(\cat C^c)$ of thick ideals of compact objects of $\cat C$.
	\end{enumerate}
Correspondingly, we obtain the following results, which are obtained in \Cref{prop-equalizer,cor-smashing-eq,thm-equualizer-thick-ideals}.
\begin{thmx}
	Let $\cat C$ be a rigidly-compactly generated tt-$\infty$-category and $A \in \CAlg(\cat C)$ a descendable commutative algebra.
	\begin{enumerate}
		\item There is a split equalizer of posets
			\[
			\begin{tikzcd}
				\Locid(\cat C) \arrow[r] & \Locid(\mod{A}(\C)) \arrow[r, shift right] \arrow[r, shift left] & \Locid(\mod{A\otimes A}(\C)).
			\end{tikzcd}
			\]
		\item There is an equalizer of posets
			\[
			\begin{tikzcd}
				\Smashid(\cat C) \arrow[r] & \Smashid(\mod{A}(\C)) \arrow[r, shift right] \arrow[r, shift left] & \Smashid(\mod{A\otimes A}(\C)).
			\end{tikzcd}
			\]
		\item If $A$ is compact, then there is an equalizer of posets
			\[
			\begin{tikzcd}
				\Thickid(\cat C^c) \arrow[r] & \Thickid(\mod{A}(\C)^c) \arrow[r, shift right] \arrow[r, shift left] & \Thickid(\mod{A\otimes A}(\C)^c).
			\end{tikzcd}
			\]
	\end{enumerate}
\end{thmx}
Recall that one formulation of the telescope conjecture for tt-categories is that the map $\sigma \colon \Thickid(\cat C^c) \to \Smashid(\cat C)$ given by $\cat J \mapsto \Locid(\cat J)$ is bijective. In \cref{thm-telescope} we deduce the following descent result for the telescope conjecture:
\begin{thmx}
	Let $\cat C$ be a rigidly-compactly generated tt-$\infty$-category and $A \in \CAlg(\cat C)$ a compact descendable commutative algebra.  If the telescope conjecture holds for $\mod{A}(\C)$, then it holds for $\C$ too. 
\end{thmx}

For example, this theorem applies to the category of modules over the 2-local connective spectrum of topological modular forms; see \cref{ex:tmf}.

Inspired by Balmer's descent result for separable commutative algebras, we also investigate the Zariski frame of principal radical thick ideals. Translating our result from frames back into topological spaces, we prove the following in \Cref{cor:coeq-of-spectral}:

\begin{thmx}\label{thmx:c}
	Let $\cat C$ be a rigidly-compactly generated tt-$\infty$-category and $A\in \CAlg(\cat C)$ a compact commutative algebra. Then the diagram induced by base change
		\[
		\begin{tikzcd}[column sep=large]
			\Spc(\mod{A\otimes A}(\C)^c )\arrow[r, shift left] \arrow[r, shift right]& \Spc(\mod{A}(\C)^c) \arrow[r] & \supp(A)
		\end{tikzcd}
		\]
	is a coequalizer in the category of spectral spaces. 
\end{thmx}
We note again that if $A$ is descendable, then $\supp(A) = \Spc(\C^{c}) $, and so in this case we obtain a coequalizer computing $\Spc(\C^{c})$. 

In comparison to Balmer's result \eqref{eq:balmer-coequalizer}, note that we have removed the requirement that the commutative algebra is separable at the expense of having to work in the category of \emph{spectral spaces} instead of topological spaces. As noted above, coequalizers in spectral spaces and topological spaces need not agree. It is a consequence of Balmer's work that under the assumption that $A$ is separable of finite degree the two coequalizers do agree (see \Cref{comparison-with-Balmer}).

Our next result involves descent for \emph{stratification} in the sense of \cite{BHS2021}. We recall that this is a condition that gives a classification of localizing ideals of $\cat C$ in terms of subsets $\Spc(\cat C^c)$. Partial results in this direction have been obtained in \cite{BCHS-costratification}. Our results can be summarized as follows; the results given in the body of the document (\Cref{thm-stratification-descent}) are slightly stronger. 
\begin{thmx}
	Let $\cat C$ be a rigidly-compactly generated tt-$\infty$-category and let $A \in \CAlg(\C)$ be descendable with base change functor $F_A \colon \C \to \mod{A}(\C)$. Suppose that $\mod{A}(\C)$ is stratified and $\Spc(\mod{A \otimes A}(\C)^c)$ is weakly noetherian. Then the following are equivalent:
	\begin{enumerate}
		\item $\C$ is stratified. 
		\item The identity $\Supp(F_A(x))=\varphi^{-1}(\Supp(x))$ holds for all $x \in \cat C$. 
		\item The diagram induced by base change
			\[
			\begin{tikzcd}[column sep=large]
				\Spc(\mod{A\otimes A}(\C)^{c} )\arrow[r, shift left] \arrow[r, shift right]& \Spc(\mod{A}(\C)^{c}) \arrow[r] & \Spc(\C^{c})
			\end{tikzcd}
			\]
			is a coequalizer in the category of sets. 
	\end{enumerate}
\end{thmx}

As an application of this result, we prove that stratification satisfies a version of Galois descent (\Cref{cor-descent-strat-fin-group}). 
\begin{thmx}\label{thmx:e}
	Let $A$ be a faithful $G$-Galois extension in a rigidly-compactly generated tt-$\infty$-category $\C$ where $G$ is a finite group. Suppose that $\mod{A}(\C)$ is stratified with noetherian Balmer spectrum. Then $\C$ is also stratified with noetherian Balmer spectrum given by $\Spc(\mod{A}(\C)^c)/G$.
\end{thmx}
Using this and \cite{BCHNP2023} we show in \Cref{exa:ko-g-stratification} that the category $\mod{KO_G}(\Sp_G)$ is stratified with noetherian spectrum $\Spc(\mod{KO_G}(\Sp_G)^c) \cong \spec(R(G))/C_2$, where $KO_G$ denotes real equivariant $K$-theory, and $G$ is a finite group. 

In the case of a Galois extension of commutative ring spectra, \Cref{thmx:e} admits the following strengthening (\cref{descent-strat-compact-lie}):

\begin{thmx}
	Let $A \to B$ be a faithful $G$-Galois extension of commutative ring spectra where $G$ is a compact Lie group. If $\mod{B}$ is stratified with noetherian Balmer spectrum, then so is $\mod{A}$.
\end{thmx}

\subsection{Set theoretic considerations}
We fix three uncountable, strongly inaccessible cardinals $\kappa_0 < \kappa_1 <\kappa_2$ and corresponding universes $\U_{\kappa_0} \in \U_{\kappa_1} \in \U_{\kappa_2}$. A set, simplicial set, category, etc., will be said to be \emph{small} if it is contained in $\U_{\kappa_0}$; will be said to be \emph{large} if it is contained in the universe $\U_{\kappa_1}$; and will be said to be \emph{very large} if it is contained in the universe $\U_{\kappa_2}$. We will say that a set, simplicial set, category, etc., is essentially small if it is equivalent (in the appropriate sense) to a small one. 

\subsection*{Acknowledgements} 
We thank Scott Balchin for helpful conversations related to this paper.
We are also grateful to the organizers of the Abel Symposium on Triangulated Categories in Representation Theory and Beyond for their invitation to make a contribution to the proceedings.

\section{Split equalizers}

We begin with some preliminary lemmas concerning split equalizers.

\begin{definition}
	A \emph{fork diagram} in a category $\cat C$ is a diagram
	\[
		 \begin{tikzcd}
		 X_0 \arrow[r, "f"] & X_1 \arrow[r, shift left, "\alpha"] 
		 \arrow[r, shift right, "\beta"'] & X_2
		 \end{tikzcd}
	\]
	in which $\alpha\circ f=\beta\circ f$.
	A \emph{split fork diagram} is such a diagram equipped with two additional arrows
	\[
		\begin{tikzcd}
		X_0 \arrow[r, "f"] & X_1  \ar[l,bend left=60,"u"] \arrow[r, shift left, "\alpha"] \arrow[r, shift right, "\beta"'] & X_2 \ar[l,bend left=60, "v"]
		\end{tikzcd}
	\]
	in which $u\circ f=\id$, $v\circ \beta=\id$ and $v\circ \alpha=f\circ u$.
\end{definition}

\begin{remark}\label{rem:split-fork-is-equalizer}
	Every split fork diagram 
	\[
	\begin{tikzcd}
		X_0 \ar[r,"f"] \ar[r] &X_1 \arrow[r, shift left, "\alpha"] \arrow[r, shift right, "\beta"'] & X_2
	\end{tikzcd}
	\]
	is an equalizer; see \cite[Section VI.6]{Maclane}. In fact, a split fork diagram is an absolute equalizer; that is, an equalizer which is preserved by all functors. We use \emph{split equalizer} as a synonym for \emph{split fork diagram}.
\end{remark}

\begin{lemma}\label{lem-mega-lemma}
	Consider a diagram of sets
	\[
	\begin{tikzcd}
		X_0 \arrow[r, "f"] & X_1 \arrow[r, shift left, "\alpha"] 
		\arrow[r, shift right, "\beta"'] & X_2 \\
		Y_0 \arrow[u,"i_0"] \arrow[r,"g"] & Y_1 \arrow[u, "i_1"]\arrow[r, shift left, "\gamma"] 
		\arrow[r, shift right, "\delta"'] & Y_2 \arrow[u, "i_2"]
	\end{tikzcd}
	\]
	in which
	\begin{itemize}
		\item[(i)] The top fork is split: there exists $u\colon X_1 \to X_0$ and $v \colon X_2 \to X_1$ such that $u\circ f=\id$, $v\circ \beta=\id$ and $v\circ \alpha=f\circ u$;
		\item[(ii)] $i_0$ and $i_1$ are injective;
		\item[(iii)] $f \circ i_0=i_1\circ g$;
		\item[(iv)] $\alpha\circ i_1=i_2\circ \gamma$ and $\beta\circ i_1=i_2\circ \delta$.
	\end{itemize}
	Then the following are equivalent:
	\begin{itemize}
		\item[(a)] The bottom fork in an equalizer;
		\item[(b)] For all $y_1\in Y_1$ such that $\gamma(y_1)=\delta(y_1)$, there exists $y_0\in Y_0$ such that $u(i_1(y_1))=i_0(y_0)$.
	\end{itemize}
\end{lemma}

\begin{proof}
	$(a)\Rightarrow (b)$: If $y_1\in Y_1$ is such that $\gamma(y_1)=\delta(y_1)$ then by (a) there exists $y_0 \in Y_0$ such that $g(y_0)=y_1$. We claim that $u(i_1(y_1))=i_0(y_0)$ giving (b). Since $\gamma(y_1)=\delta(y_1)$ then $i_2(\gamma(y_1))=i_2(\delta(y_1)$ and so $\alpha(i_1(y_1))=\beta(i_1(y_1))$ by condition~(iv). Applying~$v$ to this identity and using that the fork is split we find 
	\[
		f(u(i_1(y_1)))=v(\alpha(i_1(y_1)))=v(\beta(i_1(y_1)))=i_1(y_1).
	\]
	Note that the right hand side can be rewritten as $i_1(y_1)=i_1(g(y_0))=f(i_0(y_0))$ by (iii). Note that $f$ is injectve by (i) so we conclude that $u(i_1(y_1))=i_0(y_0)$ giving~(b).
 
	$(b)\Rightarrow (a)$: Let $y_1 \in Y_1$ be such that $\gamma(y_1)=\delta(y_1)$. We need to show that there exists unique $y_0 \in Y_0$ such that $g(y_0)=y_1$.  Uniqueness follows from the fact that $g$ is injective since the composite $i_1g=fi_0$ is so by combining (i) and (ii). Let us construct such $y_0$. From $\gamma(y_1)=\delta(y_1)$ and (iv) we deduce that $\alpha(i_1(y_1))=\beta(i_1(y_1))$. Applying $v$ to this and using the identities of the split fork we find that 
	\begin{equation}\label{eqqq}
		f(u(i_1(y_1)))=i_1(y_1)
	\end{equation}
	as in the previous paragraph. By condition (b) there exists $y_0 \in Y_0$ such that $i_0(y_0)=u(i_1(y_1))$. We claim that such $y_0$ does the job, namely $g(y_0)=y_1$. Indeed
	\[
		 i_1(g(y_0))=f(i_0(y_0))=f(u(i_1(y_1)))\overset{\eqref{eqqq}}{=}i_1(y_1)
	\]
	so we conclude by injectivity of $i_1$.
\end{proof}

\begin{lemma}\label{lem-diagram-chase}
	Consider a diagram of sets
	\[
	\begin{tikzcd}
		X_0 \arrow[r, "f"]\arrow[d,shift left, "r_0"] & X_1 \arrow[d,shift left, "r_1"] \arrow[r, shift left, "\alpha"] \arrow[r, shift right, "\beta"'] & X_2 \arrow[d,shift left,"r_2"]\\ Y_0 \arrow[u,shift left,"i_0"] \arrow[r,"g"] & Y_1 \arrow[u,shift left, "i_1"]\arrow[r, shift left, "\gamma"] \arrow[r, shift right, "\delta"'] & Y_2 \arrow[u,shift left, "i_2"]
	\end{tikzcd}
	\]
	in which 
	\begin{itemize}
		\item[(i)] The top fork is an equalizer and $\gamma \circ g=\delta \circ g$;
		\item[(ii)] $r_0 \circ i_0=\id$, $r_1 \circ i_1=\id$, and $r_2 \circ i_2=\id$;
		\item[(iii)] $i_1$ is bijective (hence with inverse $r_1$);
		\item[(iv)] $f \circ i_0=i_1 \circ g$, $\alpha \circ i_1=i_2 \circ \gamma$ and $\beta \circ i_1=i_2 \circ \delta$. 
	\end{itemize}
	Then the following are equivalent:
	\begin{itemize}
		\item[(a)] $r_0$ is bijective;
		\item[(b)] $g \circ r_0 =r_1 \circ f$;
		\item[(c)] The bottom fork is an equalizer. 
	\end{itemize}
\end{lemma}

\begin{proof}
	$(a) \Rightarrow (b)$: Using (ii) and (iv) we see that $g=r_1 \circ i_1 \circ g= r_1 \circ f\circ i_0$. Now combining (ii)+(a) we conclude that $r_0$ has inverse~$i_0$. So $g\circ r_0 =r_1 \circ f\circ i_0 \circ r_0= r_1 \circ f$ as required.  

	$(b)\Rightarrow (c)$: Given $y_1 \in Y_1$ such that $\gamma(y_1)=\delta(y_1)$ we need to show that there exists a unique $y_0 \in Y_0$ such that $g(y_0)=y_1$. Note that $g$ is injective since the composite $i_1\circ g=f\circ i_0$ is so by assumption. Therefore if it exists $y_0$ is necessarily unique. Note that $i_2(\gamma(y_1))=i_2(\delta(y_1))$ so by (iii) we see that $\alpha(i_1(y_1))=\beta(i_1(y_1))$. Since the top diagram is an equalizer we find unique $x_0\in X_0$ such that $f(x_0)=i_1(y_1).$ We claim that $y_0\coloneqq r_0(x_0)$ does the job. Indeed
	\[
		g(y_0)=g(r_0(x_0))\overset{(b)}{=}r_1f(x_0)=r_1i_1(y_1)\overset{(ii)}{=}y_1.
	\]

	$(c)\Rightarrow (a)$: By condition (ii) it suffices to show that $r_0$ is injective or equivalently that $i_0$ is surjective. So pick $x_0\in X_0$; we will show that $x_0$ is in the image of $i_0$. By assumption $\alpha \circ i_1=i_2 \circ \gamma$ and $\beta \circ i_1=i_2 \circ \delta$ so by precomposing with $r_1$ and using (iii) we find $\alpha=i_2 \circ \gamma \circ r_1$ and $\beta=i_2 \circ \delta \circ r_1$. Clearly $r_2(\alpha(f(x_0)))=r_2(\beta(f(x_0)))$ so
	\begin{align*}
		 \gamma(r_1(f(x_0)))=r_2(i_2(\gamma(r_1(f(x_0)))))&=r_2(\alpha(f(x_0)))\\
		 &=r_2(\beta(f(x_0)))\\
		 &=r_2(i_2(\delta(r_1(f(x_0)))))=\delta(r_1(f(x_0)))
	\end{align*}
	by inserting the identities for $\alpha$ and $\beta$. Then by (c) we find $y_0\in Y_0$ such that $g(y_0)=r_1(f (x_0))$. We claim that $i_0(y_0)=x_0$. Indeed $f(i_0(y_0))=i_1(g(y_0))=i_1(r_1(f(x_0)))=f(x_0)$ and by injectivity of $f$ we conclude that $i_0(y_0)=x_0$.
\end{proof}

\section{Stone duality}

We briefly recall some background material on lattices and Stone duality. Our main reference is \cite{Johnstone} where the reader can find a more elaborate discussion.

\subsection{Lattices}
A partial ordered set $(A,\leq )$ is a \emph{lattice} if any two-element subset $\{a,b\}\subseteq A$ has a join (or least upper bound) $a \vee b$, and a meet (or greatest lower bound) $a\wedge b$. We also require that it contains a least element $0$ and a greatest element~$1$. A morphism of lattices $f\colon A \to B$ is a function satisfying $f(a\wedge a')=f(a)\wedge f(a')$ and $f(a\vee a')=f(a)\vee f(a')$ for all $a,a'\in A$, and also $f(0)=0$ and $f(1)=1$. We denote the category of lattices by $\Lat$.

\begin{remark}
	Many authors do not require a lattice to have $0$ and $1$ and would call the above notion a \emph{bounded} lattice. We follow the terminology of \cite{Johnstone}.
\end{remark}

\begin{remark}
	Any lattice morphism is necessarily order-preserving (since $a \leq a'$ if and only if $a \vee a'=a'$) but not all order-preserving maps are lattice morphisms. For example, let $\CP$ be the power set of the two-element set $\{a,b\}$ and let $\CQ$ be the totally-ordered set $\{0 < 1 < 2\}$. Then the cardinality map $c\colon \CP \to \CQ$ is order-preserving but not a lattice morphism since $c(\{ a\} \vee \{b \})=c(\{ a,b\})=2$ and yet $c(\{ a\}) \vee c(\{ b\})=1 \vee 1 =1$.
\end{remark}

\begin{remark}
	According to the above remark, the forgetful functor $\Lat \to \Pos$ from lattices to posets is faithful but not full. On the other hand, one readily checks that if $f\colon A \to B$ is a bijective lattice morphism, then $f^{-1}\colon B\to A$ is a lattice morphism. Hence, lattice isomorphisms coincide with bijective lattice morphisms. In particular, the forgetful functors $\Lat \to \Pos$ and $\Lat \to \Set$ reflect isomorphisms.
\end{remark}

\subsection{Distributive lattices}
A lattice $(A,\le)$ is said to be \emph{distributive} if \[a \wedge (b \vee c)=(a\wedge b)\vee (a\wedge c)\] for all $a,b,c \in A$. There are a number of equivalent characterizations; see \cite[I.1.5]{Johnstone}, \cite[Section I.6]{Birkhoff67}, and \cite[Section II.5]{BalbesDwinger74}. The distributive lattices form a full subcategory $\DLat$ of the category of lattices $\Lat$.

\subsection{Coherent frames}
	A lattice $A$ is \emph{complete} if \emph{every} subset $S \subseteq A$ has both a meet $\bigwedge S$ and a join $\bigvee S$ in $A$. A \emph{frame} is a complete lattice $F$ in which finite meets distribute over arbitrary joins:
	\[
		a \wedge \bigvee_{s\in S}s = \bigvee_{s\in S} (a \wedge s)
	\]
	for any $a \in F$ and $S \subseteq F$. A \emph{frame map} is a lattice morphism that preserves arbitrary joins. We denote the category of frames and frame maps by $\Frm$. An element $c$ of a frame $F$ is said to be \emph{finite} if whenever $c\leq \bigvee_{s\in S} s$ for some subset $S\subseteq F$, there exists a finite subset $K \subseteq S$ such that $c \leq \bigvee_{s\in K}s$. A frame $F$ is \emph{coherent} if every element can be expressed as a join of finite elements and the finite elements form a (distributive) lattice $F^f$. This amounts to requiring that $1$ is finite and that the meet of two finite elements is finite. (See \cite[Section II.3]{Johnstone}.) A frame map is \emph{coherent} if it takes finite elements to finite elements. We denote the category of coherent frames and coherent frame maps by $\CFr$. By definition we have a functor $(-)^f \colon \CFr \to \DLat$ which sends a coherent frame to its finite elements.

\subsection{Spectral spaces}
	A topological space is \emph{sober} if it is $T_0$ and every nonempty irreducible closed set is the closure of a (unique) point. A \emph{spectral space} is a quasi-compact sober space in which the quasi-compact open subsets are closed under finite intersection and form a basis for the topology. A \emph{spectral map} between spectral spaces is a continuous map for which the inverse image of any quasi-compact open subset is quasi-compact. The category of spectral spaces and spectral maps is denoted by $\Spec$. We have a functor \[ \Omega\colon\Spec\op \to \CFr\] which sends a spectral space $X$ to the coherent frame $\Omega(X)$ of open subsets of $X$. Our terminology is due to \cite{Hochster}. Note that spectral spaces are called ``coherent spaces'' in \cite{Johnstone} due to the following theorem:

\begin{theorem}[Stone duality]\label{thm-stone-duality}
	The functors
	\[
		\Omega\colon\Spec\op \to \CFr \qquad\text{ and }\qquad (-)^f\colon \CFr \to \DLat
	\]
	are equivalences of categories.
\end{theorem}

\begin{proof}
	See~\cite[Corollary II.3.4]{Johnstone}.
\end{proof}

\begin{remark}
	In particular, we have an equivalence 
	\begin{equation}\label{eq:Spec-to-DLat}
		\Spec\op \xrightarrow{\simeq} \DLat
	\end{equation}
	which sends a spectral space $X$ to the distributive lattice of quasi-compact open subsets of $X$. We now describe a quasi-inverse to this equivalence:
\end{remark}

\subsection{The spectrum of a distributive lattice}
	A subset $I$ of a distributive lattice~$A$ is an \emph{ideal} if: 
	\begin{itemize}
		\item $I$ is non-empty;
		\item if $a\in I$ and $b\in A$ are such that $b\leq a$, then $b\in I$;
		\item for all $a,b \in I$, the element $a \vee b \in I$.
	\end{itemize} 
	Moreover $I$ is \emph{prime} if in addition it is a proper subset of $A$ and satisfies:
	\begin{itemize}
		 \item if $a \wedge b \in I$ then either $a\in I$ or $b \in I$.
	\end{itemize}
	The set of prime ideals of $A$ is denoted by $\spec(A)$. We endow this set with a topology by taking as basis the subsets $d(a)\coloneqq \{I \in \spec(A) \mid a \not\in I\}$ for all $a\in A$. We refer to $\spec(A)$ as the \emph{spectrum} of the distributive lattice $A$. Given a lattice map $f\colon A \to B$, we define a map $\spec(f)\colon \spec(B) \to \spec(A)$ by sending $I$ to $f^{-1}(I)$. One can check that this defines a functor
	\[
		\spec\colon \DLat\op \to \Spec
	\]
which is quasi-inverse to the equivalence \eqref{eq:Spec-to-DLat}.
  
\begin{remark}[Hochster duality]\label{rem:hochster-duality}
	Every spectral space $X$ can be equipped with a ``Hochster dual'' spectral topology; see \cite[Proposition 8]{Hochster}. Under the equivalence $\Spec\op \cong \DLat$ this amounts to sending a distributive lattice to its opposite lattice (which categorically amounts to taking the opposite category). Note that the ``Hochster dual topology'' is called the ``inverse topology'' in \cite{Spectralbook}.
\end{remark}

\begin{remark}
	The equivalence $\Spec\op \cong \DLat$ restricts to an equivalence between profinite spaces (a.k.a.~``Stone spaces'') and Boolean lattices (a.k.a.~``Boolean algebras). This is the original manifestation of Stone duality from \cite{Stone37}.
\end{remark}

Finally, we recall the following result:

\begin{lemma}\label{lem-forgetful-creates-limits}
	The forgetful functors $\Lat \to \Set$, $\DLat \to \Set$ and $\Frm \to\Set$ create all limits.
\end{lemma}

\begin{proof}
	Johnstone~\cite[I.3.8]{Johnstone} defines a category to be \emph{algebraic} if the forgetful functor to sets is monadic and \emph{equationally presentable} if its objects can be described by (a proper class of) operations and equations. Note that $\Lat, \DLat$ and $\Frm$ are equationally presentable and so they admits all limits by~\cite{Johnstone}*{Proposition I.3.8}. Also, $\Lat$ and $\DLat$ are algebraic by the discussion in~\cite[I.3.7]{Johnstone} whereas $\Frm$ is algebraic by~\cite[II.1.2]{Johnstone}. It is only left to note that monadic functors create limits; see for instance~\cite{Maclane}*{Exercise~2, p.~142}.
\end{proof}

\begin{remark}
	In contrast, the forgetful functor $\Spec \to \Set$ does not preserve colimits. For example, under Stone duality, the one-point space corresponds to the Boolean algebra $\mathbbm{2} \coloneqq \{0 < 1\}$. The product of countably infinitely many copies of~$\mathbbm{2}$ in $\DLat$ corresponds under Stone duality to the Stone-\v{C}ech compactification $\beta \mathbb{N}$ of the natural numbers. Thus, $\beta \mathbb{N}$ is the coproduct in $\Spec$ of countably infinitely many copies of the one-point space. Since $\beta \mathbb{N}$ is uncountable, we deduce that the forgetful functor $\Spec \to \Set$ does not preserve this infinite coproduct. The extent to which $\Spec \to \Set$ and $\Spec \to \Top$ preserve colimits turns out to be highly relevant for our questions. This is what we turn to next.
\end{remark}

\section{Spectral and topological coequalizers}

Colimits in spectral spaces are in general different from colimits taken in the category of topological spaces as the next example shows. Our aim is to describe conditions under which the coequalizer of a diagram in spectral spaces and spectral maps is homeomorphic to the topological coequalizer. 

\begin{example}\label{ex:spectral-vs-top-colimit}
	This example is taken from~\cite{Spectralbook}*{Example 10.2.9 (i)}. Recall that finite spectral spaces are equivalent to finite posets; see~\cite{Spectralbook}*{\S 1.1.16}. Under this equivalence spectral maps coincide with monotonic maps. Then consider the following diagram of finite posets:
	\[
	\begin{tikzcd}
		{\{0, 1\} }\arrow[r] \arrow[d] & {\{0 > 1\}} \\
		{\{1>0\}} &
	\end{tikzcd}
	\]
	The spectral pushout has only one point, whereas the topological pushout is indiscrete with two points, which is not a $T_0$-space.
\end{example}

\subsection{Spectral reflection}
We now recall some properties of the spectral reflection functor $S\colon \Top \rightarrow \Spec$ which has a unit denoted by  $S_X\colon X\rightarrow S(X)$ for each $X\in \Top$ (following \cite[\S 11]{Spectralbook} and \cite{Schwartz17}). This functor is the left adjoint of the forgetful functor $U\colon \Spec \rightarrow \Top$, and therefore it commutes with colimits (in particular, with coequalizers). 

It is worth recalling that the category of spectral spaces is not a full subcategory of the category of topological spaces. This implies that even if $X$ is a spectral space, the unit $S_X\colon X\rightarrow S(X)$ is not necessarily a homeomorphism. Nevertheless, the behaviour of $S_X$ is well-understood.

\begin{theorem}[Schwartz]\label{thm:spectral-reflection}
	Let $X$ be a topological space and let $S_X \colon X \to S(X)$ be its spectral reflection. Then
	\begin{enumerate}
		\item $S_X$ is a quasi-homeomorphism onto its image.
		\item $S_X$ is injective if and only if $X$ is $T_0$.
		\item $S_X$ is a homeomorphism onto its image if and only if $X$ is $T_0$.
		\item The corestriction $S_X\colon X\rightarrow S_X(X)$ is the $T_0$-reflection of $X$. 
		\item $S_X$ is a homeomorphism if and only if $X$ is a noetherian spectral space.
		\item $S_X$ is surjective if and only if the $T_0$-reflection of $X$ is a noetherian spectral space.
	\end{enumerate}
\end{theorem}

\begin{proof}
	These results are proved in \cite{Schwartz17}. Alternatively, see \cite{Spectralbook}*{Theorems 11.1.3 and 11.1.12, and \S 11.4.3}. 
\end{proof}

\subsection{$T_0$-reflection}
The fully faithful embedding of $T_0$-spaces in $\Top$ also has a left adjoint, which is quite simple to describe: The $T_0$-reflection (or ``Kolmogorov quotient'') of a topological space $X$ is the quotient $X \to \KQ(X)\coloneqq X/{\equiv}$ by the equivalence relation on $X$ defined by $x\equiv y$ iff $\overbar{\{x\}}=\overbar{\{y\}}$; cf.~\cite[\S 6.1.4]{Spectralbook}. In particular it commutes with colimits and the unit $X\rightarrow \KQ(X)$ is a homeomorphism if $X$ is a $T_0$-space.

\begin{lemma}\label{lem:KQ-spec}
	The spectral reflection of a topological space coincides with the spectral reflection of its $T_0$-reflection:
	\[\begin{tikzcd}
		X \ar[d,"S_X"] \ar[r,"id"] & \KQ(X) \ar[d,"S_{\KQ(X)}"]\\
		S(X) \ar[r,"\cong"] & S(\KQ(X)).
	\end{tikzcd}\]
\end{lemma}

\begin{proof}
	This is a routine exercise using the universal properties of the $T_0$-reflection and the spectral reflection, together with the fact that a surjective continuous map, such as $X \to \KQ(X)$, is an epimorphism in the category of topological spaces.
\end{proof}

\begin{remark}\label{rem:coeq-top-spec}
	Suppose we are given a coequalizer diagram
	\begin{equation}\label{eq:coeq-S}
	\begin{tikzcd}
		Z \arrow[r, shift left, "g"] \arrow[r, shift right, "h"'] & Y \arrow[r, "\varphi_S"] & T^{\Spec}
	\end{tikzcd}
	\end{equation}
	in the category of spectral spaces (that is, $Z,Y,g,h$ are spectral spaces and spectral maps). We can also consider the following coequalizer in the category of topological spaces, 
	\begin{equation}\label{eq:coeq-T}
	\begin{tikzcd}
		UZ \arrow[r, shift left, "Ug"] \arrow[r, shift right, "Uh"'] & UY \arrow[r, "\varphi_T"] & T^{\Top}.
	\end{tikzcd}
	\end{equation}
	Both maps $\varphi_S$ and $\varphi_T$ are epimorphisms in $\Spec$ and $\Top$ respectively, and therefore surjective. Applying the forgetful functor to the diagram \eqref{eq:coeq-S}, we obtain a comparison map $p\colon T^{\Top}\rightarrow UT^{\Spec}$ such that $p\circ \varphi_T=U\varphi_S$.  
\end{remark}

\begin{remark}\label{rem:comparison map is closed}
	In the situation of \cref{rem:coeq-top-spec}, first note that $p$ is surjective since $\varphi_S$ is so. Moreover, if $\varphi_S$ is closed then so is the comparison map $p$: let $K\subseteq T^\Top$ be a closed subset; then $p(K)=p(\varphi_T(\varphi_T^{-1}(K)))=\varphi_S(\varphi^{-1}_T(K))$ is also closed.
\end{remark}

In some favourable cases the comparison map is automatically a homeomorphism. 

\begin{proposition}\label{prop-Ttop-spectral}
	In the situation of \cref{rem:coeq-top-spec}, the following are equivalent:
	\begin{itemize}
	    \item[(a)] $T^{\Top}$ is a spectral space and $\varphi_T$ is a spectral map. 
	    \item[(b)] The comparison map $p\colon T^{\Top}\rightarrow UT^{\Spec}$ is a homeomorphism.
	\end{itemize}
\end{proposition}

\begin{proof}
	By definition, there is a commutative diagram in the category of topological spaces
	\[
	\begin{tikzcd}
		UZ \sqcup UY \arrow[r,"{(g,1)}"] \arrow[d,"{(h,1)}"'] & UY \arrow[d,"\varphi_T"] \arrow[ddr,bend left, "\varphi_S"] &\\
		UY \arrow[r,"\varphi_T"] \arrow[drr,bend right, "\varphi_S"] & T^{\Top} \arrow[dr,"p"] & \\
		&  & UT^{\Spec}.
	\end{tikzcd}
	\]
	where the top square is a pushout. The fact that (a) implies (b) now follows from~\cite{Spectralbook}*{Corollary 10.2.6}. 
	If (b) holds, then $\varphi_S=p \circ \varphi_T$ together with the fact that $\varphi_S$ is spectral implies that $\varphi_T$ is also spectral.
\end{proof}

\begin{remark}
	Since the spectral reflection commutes with colimits we obtain a new coequalizer diagram
	\[
	\begin{tikzcd}
		S(UZ) \arrow[r, shift left, "S(Ug)"] \arrow[r, shift right, "S(Uh)"'] & S(UY) \arrow[r, "S(\varphi_T)"] & S(T^{\Top})
	\end{tikzcd}
	\]
	in the category of spectral spaces. The spectral reflection induces a map from $T^{\Spec}$ to $S(T^{\Top})$ which is not necessarily a homeomorphism. That is, the spectral coequalizer need not be the spectral reflection of the topological coequalizer. 
\end{remark}

\begin{proposition}\label{prop:T-is-T_0}
	In the situation of \cref{rem:coeq-top-spec}, suppose that $Y$ and $Z$ are noetherian spectral spaces. Then the comparison map $p\colon T^{\Top} \rightarrow UT^{\Spec}$ is the $T_0$-reflection of $T^{\Top}$. Therefore, in this case, the map $p$ is an homeomorphism if and only if $T^{\Top}$ is a $T_0$-space.
\end{proposition}

\begin{proof}
	Since $Y$ and $Z$ are noetherian spectral spaces, the coequalizer $T^{\Spec}$ is also a noetherian spectral space. Moreover, the spectral reflection is an isomorphism in $\Spec$ by \Cref{thm:spectral-reflection}. Then consider the following commutative diagram of topological spaces:
	\[\begin{tikzcd}
		Z \ar[r,shift left,"g"] \ar[r,shift right,"h"'] \ar[d,"S_Z"',"\cong"] & Y  \ar[d,"S_Y"',"\cong"] \ar[r,"\varphi_S"'] & T^{Spec} \ar[d,dotted,"\exists !\theta","\cong"'] \\
		S(UZ) \ar[r,shift left,"S(g)"] \ar[r,shift right,"S(h)"'] & S(UY) \ar[r,"S(\varphi_T)"]  & S(T^{Top}) 
	\end{tikzcd}\]
	We deduce $\theta \colon T^{\Spec} \xrightarrow{\cong} S(T^{\Top})$ is a homeomorphism, by uniqueness of coequalizers. Moreover, since $p\circ \varphi_T=U\varphi_S$, applying the forgetful functor to the first row in the diagram we have the composite $\theta \circ U\varphi_S= (\theta\circ p)\circ \varphi_T = S(\varphi_T)\circ S_Y$. But, the naturality of the spectral reflection $S(\varphi_T)\circ S_Y= S_{T^{\Top}}\circ \varphi_T$ implies $\theta\circ p = S_{T^\Top}$ by uniqueness of the factorization. In other words, the map $p$ can be identified, up to homeomorphism, with the spectral reflection map $T^{\Top} \to S(T^{\Top})$. The result then follows from \Cref{thm:spectral-reflection}(d). 
\end{proof}

\begin{remark}\label{rem-spectral-quot}
	The coequalizer in the category of spectral spaces is the spectral quotient of $Y$ by the equivalence relation generated by 
	\[
	\{(f(z),g(z))|z\in Z\}\subseteq Y\times Y
	\]
	see \cite[Proposition 6.1.6]{Spectralbook}. For a discussion on quotients of spectral spaces see~\cite{Fargues}*{\S 1.7} and also~\cite{Scholze2017}*{Lemma 2.9}. Therefore it will suffice to give conditions under which the spectral quotient agrees with the topological quotient. This is what we discuss next. 
\end{remark}

\begin{notation}\label{not-spectral-quot}
    Let $X$ be a spectral space, $R$ an equivalence relation on $X$ and let $p\colon X \to X/R$ be the quotient map. If $A\subseteq X$, we denote $R^{-1}(A)=p^{-1}(p(A))$, the smallest $R$-invariant subset containing $A$. Note that $R^{-1}(x)=[x]$, the equivalence class of $x\in X$. Let $p_S\colon X\rightarrow X//R$ be the spectral quotient map. By the universal property of the quotient space, there is a continuous comparison map $c\colon X/R\rightarrow X//R$ making the following diagram commutative
	\begin{equation}\label{eq:comparison-quotient}
	\begin{tikzcd}
			X  \ar[r,"p_S"] \ar[dr,"p"']  & X//R\\
			 & X/R. \ar[u,"c"']\\
	\end{tikzcd}
	\end{equation}
\end{notation}

\begin{proposition}\label{prop:spectralquotient}
	In the situation of \Cref{not-spectral-quot}, assume that $p_S\colon X\rightarrow X//R$ is closed and $p_S^{-1}(p_S(x))\subseteq X$ a $T_1$-space. If $p_S^{-1}(p_S(x))\subseteq R^{-1}(\overline{\{x\}})$ then the comparison map $c\colon X/R\rightarrow X//R$ is an homeomorphism.
\end{proposition}
\begin{proof}
	Note that all maps in the diagram \eqref{eq:comparison-quotient} are surjective since $p_S$ and $p$ are so. Moreover $c$ is closed because $p_S$ is closed by assumption (see \cref{rem:comparison map is closed}). It remains to show injectivity of $c$. For that we will prove that $p_S^{-1}(p_S(x))=[x]$. 

	We have inclusions $[x]\subseteq p_S^{-1}(p_S(x))\subseteq R^{-1}(\overline{\{x\}})$. Let $y\in p_S^{-1}(p_S(x))$. Then there exists $z\in \overline{\{x\}}$ with $[y]=[z]$; in particular $p_S(z)=p_S(y)=p_S(x)$. Since $p_S^{-1}(p_S(x))\subseteq X$ is a $T_1$-space, $\overline{\{x\}}\cap p_S^{-1}(p_S(x))=\{x\}$ implies that $z=x$. Finally, $[y]=[x]$.
\end{proof}

\begin{corollary}\label{cor:comparison_coequalizers}
    In the situation of \cref{rem:coeq-top-spec}, let $R$ be the equivalence relation generated by the coequalizer diagram. Assume that $\varphi_S$ is closed and that $\varphi_S^{-1}(\varphi_S(y))$ is a $T_1$-space for all $y\in Y$. If $\varphi_S^{-1}(\varphi_S(y))\subseteq R^{-1}(\overline{\{y\}})$ for all $y\in Y$ then the comparison map $c\colon T^\Top \rightarrow UT^\Spec$ is an homeomorphism.
\end{corollary}

\begin{proof}
Combine \Cref{rem-spectral-quot} with \Cref{prop:spectralquotient}.
\end{proof}

\begin{example}
	Let $X=\Spec(\mathbb{Z})_1\sqcup \Spec(\mathbb{Z})_2$ be the disjoint union of two copies of $\Spec(\mathbb{Z})$ with the equivalence relation generated by $(p)_1\sim (p)_2$ if $p\neq 0$. Note that $[(0)_1]_R\neq [(0)_2]_R$ but $\Cl_R((0)_1)=\Cl_R((0)_2)=X$. The topological quotient is not $T_0$, and the spectral quotient is $\Spec(\mathbb{Z})$. The fibers of $p_0$ are discrete, and the spectral quotient map $p_0$ is closed, but $p_0^{-1}(p_0((0)_1))=\{(0)_1,(0)_2\}$ and $(0)_2\not \in R^{-1}(\overline{\{(0)_1\}})=X\setminus \{(0)_2\}$. Note also that $R^{-1}(\overline{\{(0)_1\}})\subseteq X$ is not closed.
\end{example}

\section{Triangulated categories and higher categories}\label{sec:categories}

We now turn to the world of derived algebra and describe the types of (higher) categories we aim to study. We also take this opportunity to recall various definitions, notation and terminology.

\begin{definition}
	A \emph{tt-category} is a triangulated category equipped with a compatible closed symmetric monoidal structure in the sense of \cite[Appendix A]{HPS}. We denote the internal hom by $\ihom$ and the dual functor by $D\coloneqq \ihom(-,\unit)$.
\end{definition}

\begin{definition}\label{def:thick-subcategories}
	Let $\cat T$ be a tt-category.
	\begin{enumerate}
		\item A \emph{thick subcategory} of $\cat T$ is a (full, replete) triangulated subcategory which is thick (i.e., closed under direct summands).
		\item A \emph{thick ideal} of $\cat T$ is a thick subcategory $\cat I \subseteq \cat T$ such that $\cat T \otimes \cat I \subseteq \cat I$; that is, 
			if $a \in \cat T$ and $b \in \cat I$, then $a\otimes b \in \cat I$.
		\item A \emph{radical thick ideal} of $\cat T$ is a thick ideal $\cat I \subseteq \cat T$ with the property that if $a \in \cat T$ satisfies $a^{\otimes k} \in \cat I$ for some $k \ge 1$, then $a \in \cat I$.
		\item We write $\thick{\cat E}$, $\thickid{\cat E}$, and $\sqrt{\cat E}$ for the thick subcategory, thick ideal, and radical thick ideal generated by a collection of objects $\cat E \subseteq \cat T$.
	\end{enumerate}
\end{definition}

\begin{definition}\label{def:dual-rigid}
	Let $\cat T$ be a tt-category.
	\begin{enumerate}
		\item We say that $x \in \cat T$ is \emph{dualizable} if  for all $y \in \cat T$, the canonical map $Dx \otimes y \to \ihom(x,y)$ is an isomorphism in $\cat T$. 
		\item We say that $\cat T$ is \emph{rigid} if every object of $\cat T$ is dualizable.
		\item We denote the full subcategory of dualizable objects by $\cat T^d$. It is a rigid tt-subcategory of $\cat T$.
	\end{enumerate}
\end{definition}

\begin{remark}\label{rem-all-ideal-radical}
	If $\cat T$ is rigid, then any thick ideal is radical; see~\cite[Prop.~2.4]{Balmer2}.
\end{remark}

\begin{definition}
	A \emph{big tt-category} is a tt-category $\cat T$ which admits coproducts. 
\end{definition}

\begin{definition}
	Let $\cat T$ be a big tt-category.
	\begin{enumerate}
		\item A \emph{localizing subcategory} of $\cat T$ is a thick subcategory which is closed under coproducts.
		\item A \emph{localizing ideal} of $\cat T$ is a thick ideal which is closed under coproducts.
		\item We write $\loc{\cat E}$ and $\locid{\cat E}$ for the localizing subcategory and localizing ideal generated by a collection of objects $\cat E \subseteq \cat T$.
	\end{enumerate}
\end{definition}

\begin{definition}\label{def:big-tt-defns}
	Let $\cat T$ be a big tt-category.
	\begin{enumerate}
		\item An object $x \in \cat T$ is \emph{compact} if $\Hom_{\cat T}(x,-)\colon\cat T \to \Ab$ preserves coproducts.
		\item We denote the full subcategory of compact objects by $\cat T^c$. It is a thick triangulated subcategory of~$\cat T$.
		\item We say that $\cat T$ is \emph{compactly generated} if there exists a set $\cat G \subseteq \cat T^c$ such that 
			$\loc{\cat G} = \cat T$.
		\item We say that $\cat T$ is \emph{rigidly-compactly generated} if $\cat T^c=\cat T^d$
			and $\cat T$ is compactly generated.
	\end{enumerate}
\end{definition}

\begin{remark}\label{rem:gen-by-compact}
	For a set of compact objects $\cat G \subseteq \cat T^c$, the following statements are equivalent:
	\begin{enumerate}
		\item $\loc{\cat G} = \cat T$;
		\item For any $t \in \cat T$, if $\ihom_{\cat T}(\cat G,t)=0$ then $t=0$;
		\item For any $t \in \cat T$, if $\Hom_{\cat T}(\Sigma^n g,t)=0$ for all $g \in \cat G$ and $n \in \mathbb{Z}$, then $t=0$.
	\end{enumerate}
\end{remark}

\begin{remark}\label{rem:compact-of-compact-loc}
	A fundamental result of Neeman \cite[Lemma~2.2]{Neeman92b} asserts that if $\cat G \subseteq \cat T^c$ is a set of compact objects then $\loc{\cat G} \cap \cat T^c = \thick{\cat G}$.
\end{remark}

\begin{definition}\label{def:tt-functor}
	A \emph{tt-functor} $F\colon\cat T\to \cat S$ between tt-categories is a (strong) symmetric monoidal exact functor. Such a functor always restricts to a tt-functor $\cat T^d \to \cat S^d$ between the subcategories of dualizable objects. This is because (strong) symmetric monoidal functors preserve dualizable objects. This latter fact is most easily seen using the characterization of dualizable objects in terms of evaluation and coevaluation maps; see, e.g., \cite[Proposition III.1.9]{lms}. For brevity, a coproduct-preserving tt-functor between big tt-categories will be called a \emph{geometric functor}.
\end{definition}

We also need to introduce the $\infty$-categorical analogues of the previous definitions.

\begin{definition}\label{def:big-tt-infty}
	A \emph{big tt-$\infty$-category} is a presentable, symmetric monoidal, stable $\infty$-category $\cat C$ whose tensor $\otimes$ commutes with colimits in each variable. We note that any such $\C$ is automatically closed symmetric monoidal. A \emph{tt-$\infty$-functor} is a symmetric monoidal exact functor. The very large $\infty$-category of big tt-$\infty$-categories and colimit-preserving tt-$\infty$-functors is denoted by $\CAlg(\Pr)$.
\end{definition}

\begin{remark}\label{rem:homotopy}
	The homotopy category $\Ho(\cat C)$ of a big tt-$\infty$-category has the canonical structure of a big tt-category; see \cite[Chapter 1]{HA}. Many of the invariants and properties we are interested in concerning $\cat C$ are really invariants and properties of $\Ho(\cat C)$. For example, we can speak of thick subcategories of $\cat C$ (which are precisely the stable subcategories which are closed under retracts), thick ideals, localizing ideals, and so on. Also, an object $x \in \cat C$ dualizable iff it is dualizable as an object of $\Ho(\cat C)$; see~\cite[Section 4.6.1]{HA}.
\end{remark}

\begin{remark}\label{rem-coproduct}
	A tt-$\infty$-functor $F \colon \cat C \to \cat D$ of big tt-$\infty$-categories preserves colimits if and only if $F$ preserves coproducts; see \cite[Proposition 1.4.4.1(2)]{HA}. This latter condition is also equivalent to the induced tt-functor $\Ho(\cat C) \to \Ho(\cat D)$ of big tt-categories preserving coproducts.
\end{remark}

\begin{lemma}\label{lem-compact}
	Let $\cat C$ be a big tt-$\infty$-category. The following are equivalent for an object $x \in\C$:
	\begin{itemize}
	    \item[(a)] the functor $\Hom_{\cat C}(x,-) \colon \cat C \to \mathcal S$ to the $\infty$-category of spaces preserves filtered colimits; 
	    \item[(b)] the functor $\hom_{\cat C}(x,-)\colon \cat C \to \Sp$ to the $\infty$-category of spectra preserves filtered colimits;
	    \item[(c)] the functor $\hom_{\cat C}(x,-)\colon \C \to \Sp$ preserves coproducts;
	    \item[(d)] the functor $\Hom_{\Ho(\cat C)}^*(x,-)\colon \Ho(\cat C)\to \Ab$ into abelian groups preserves coproducts.
	\end{itemize} 
\end{lemma}

\begin{proof}
	Recall that $\Omega^\infty \hom_{\C}\simeq \Hom_\C$. Thus (b) implies (a) since $\Omega^\infty$ preserves filtered colimits. For the reverse implication, let $\{y_i\}$ be a filtered system of objects in $\C$, and consider the canonical map $\hom_{\C}(x,\colim y_i)\to \colim \hom_{\C}(x,y_i)$. This is an equivalence if and only if $\Omega^\infty\hom_{\C}(x,\Sigma^n\colim y_i)\to \colim \Omega^\infty\hom_{\C}(x,\Sigma^n y_i)$ is an equivalence for all $n\in\Z$. Using $\Omega^\infty \hom_{\C}\simeq \Hom_\C$ again, we conclude that (a) implies (b).  Note that the functor $\hom_{\cat C}(x,-)\colon \cat C \to \Sp$ is exact. Then (b) is equivalent to (c) by~\cite{HA}*{Proposition 1.4.4.1(2)}. Finally (c) is equivalent to (d) since we have a natural equivalence $\pi_n \hom_{\C}(x,-)=\Hom_{\Ho(\C)}(x, \Sigma^{-n}(-))$; see \cite{HA}*{Notation 1.1.2.17}.
\end{proof}

\begin{remark}
	Recall that an object $x \in \cat C$ which satisfies Condition $(a)$ of \cref{lem-compact} is called compact. It then follows from the same lemma that $x$ is compact in $\C$ if and only if $x$ is a compact object of $\Ho(\cat C)$ in the sense of \Cref{def:big-tt-defns}(a). Writing $\cat C^c$ for the subcategory of compact objects, we then have $\Ho(\cat C^c)=\Ho(\cat C)^c$. 
\end{remark}

\begin{remark}
	We say that a big tt-$\infty$-category $\cat C$ is compactly generated or rigidly-compactly generated when $\Ho(\cat C)$ is so.
\end{remark}

It is also convenient to have a ``small'' version of the definition:

\begin{definition}
	Let $\Cat_\infty^\perf$ denote the large $\infty$-category of essentially small, idempotent complete stable $\infty$-categories and exact functors. 
	\begin{itemize}
		\item A \emph{tt-$\infty$-category} is a commutative algebra object in $\Cat_\infty^\perf$. In concrete terms, a tt-$\infty$-category is an essentially small, idempotent complete, stable $\infty$-category equipped with a symmetric monoidal structure whose tensor is exact in each variable. This is coherent with our previous terminology, as a big tt-$\infty$-category in the sense of \cref{def:big-tt-infty} is a tt-$\infty$-category in the present sense.
		\item The large $\infty$-category of tt-$\infty$-categories and tt-$\infty$-functors is denoted by $\CAlg(\Cat_\infty^\perf)$.
		\item A tt-$\infty$-category $\cat C$ is \emph{rigid} if $\Ho (\cat C)$ is rigid (\Cref{def:dual-rigid}(b)).
	\end{itemize} 
\end{definition}

\begin{example}
	If $\C$ is a rigidly-compactly generated tt-$\infty$-category, then $\C^c=\C^d$ is a rigid tt-$\infty$-category.
\end{example}

\section{Projection formulas and base change}

We begin this section by discussing the projection formula for an adjunction of tt-categories. We then introduce the $\infty$-category of modules over a commutative algebra object and discuss various properties of the associated base change functor.

\begin{definition}
	Let $F\colon\cat T\to \cat S$ be a tt-functor which admits a right adjoint $U$. We say that the \emph{projection formula holds} if the canonical natural transformation
	\begin{equation}\label{eq:proj-formula}
		U(s)\otimes t \to U(s\otimes F(t))
	\end{equation}
	defined in \cite[(2.6)]{BalmerDellAmbrogioSanders15} is an isomorphism for all $s \in \cat S$ and $t \in \cat T$. In this case, $UF \simeq U(\unit)\otimes -$ is an isomorphism of monads \cite[Lemma 2.8]{BalmerDellAmbrogioSanders15}.
\end{definition}

\begin{lemma}\label{lem:preserve-gen}
	Let $F\colon\cat T\to \cat S$ be a tt-functor between big tt-categories which admits a right adjoint $U$.
	\begin{enumerate}
		\item If $\cat T$ is generated by dualizable objects (that is, $\cat T=\loc{\cat T^d}$) and $U$ preserves coproducts then the projection formula holds.
		\item If $U$ preserves coproducts then $F$ preserves compact objects.
		\item If $\cat T$ is compactly generated and $U$ is conservative and preserves coproducts then $\cat S=\loc{F(\cat T^c)}$.
	\end{enumerate}
\end{lemma}

\begin{proof}
	(a): This is extracted from the work of \cite{BalmerDellAmbrogioSanders16}. The projection formula~\eqref{eq:proj-formula} is always an isomorphism when the object $t$ is dualizable \cite[Proposition 3.12]{FauskHuMay03}. For a fixed $s \in \cat S$, \eqref{eq:proj-formula} is --- under our hypothesis that $U$ preserves coproducts --- a natural transformation of coproduct preserving exact functors $\cat T\to \cat T$. Hence the collection of objects on which it is an isomorphism is a localizing subcategory of $\cat T$. It contains the dualizable objects as noted, hence is everything under the hypothesis that $\cat T$ is generated by dualizable objects.

	(b): See the proof of \cite[Theorem 5.1]{Neeman96}.

	(c): This is standard and follows from part (b) and \Cref{rem:gen-by-compact}.
\end{proof}

\begin{corollary}\label{cor:gen-forced}
	Let $F\colon \cat T \to \cat S$ be a tt-functor between big tt-categories which admits a right adjoint~$U$ which is conservative and preserves coproducts. Then:
	\begin{enumerate}
		\item If $\cat T$ is compactly generated then so is $\cat S$.
		\item If $\cat T$ is rigidly-compactly generated then so is $\cat S$. 
	\end{enumerate}
\end{corollary}

\begin{proof}
	Part (a) follows from \Cref{lem:preserve-gen}. Moreover, it implies that if $\cat T$ is rigidly-compactly generated then $\cat S$ is generated by a set of compact and dualizable objects. Then to prove part (b) we need to show that $\cat S^c=\cat S^d$. Since $\unitS=F(\unitT)$ is compact, this follows from \cite[Theorem 2.1.3(d)]{HPS}.
\end{proof}

\begin{corollary}\label{cor:weakly-closed-is-closed}
	Let $F\colon \cat T \to \cat S$ be a tt-functor between big tt-categories which admits a right adjoint $U$ which is conservative and preserves coproducts. If $\cat T$ is rigidly-compactly generated and $U(\unitS)$ is compact, then $U$ preserves compact objects.
\end{corollary}

\begin{proof}
	Consider the thick subcategory $\cat S_0\coloneqq \SET{s \in \cat S^c}{U(s) \in \cat T^c} \subseteq \cat S^c$. By \Cref{lem:preserve-gen}(b),~$F$ preserves compactness. Hence $\unitS=F(\unitT)$ is compact and is thus contained in $\cat S_0$ since $U(\unitS)$ is compact by hypothesis. It remains to prove that $\cat S_0$ is a thick \emph{ideal} of $\cat S^c$. By \Cref{lem:preserve-gen}(c), $\cat S=\loc{F(\cat T^c)}$. Since $F$ preserves compactness, \Cref{rem:compact-of-compact-loc} then implies $\cat S^c = \thick{F(\cat T^c)}$. It follows that $\cat S_0$ is a thick ideal of $\cat S^c$ if and only if $F(\cat T^c)\otimes \cat S_0 \subseteq \cat S_0$. The latter is true by invoking the projection formula (\Cref{lem:preserve-gen}(a)).
\end{proof}

We will mostly use the following special case of the previous results.

\begin{proposition}\label{prop-properties-adjunction}
	Let $F \colon \cat T \to \cat S$ be a geometric functor between rigidly-compactly generated tt-categories. Then:
	\begin{enumerate}
		\item $F$ admits a right adjoint $U$;
		\item $U$ itself admits a right adjoint;
		\item $F$ preserves compact objects;
		\item The projection formula holds.
	\end{enumerate}
	If in addition $U$ is conservative and $U(\unitS)$ is compact, then $U$ preserves compact objects.
\end{proposition}

\begin{proof}
	This is \cite[Corollary 2.14 and Proposition 2.15]{BalmerDellAmbrogioSanders16}. The final claim follows from \cref{cor:weakly-closed-is-closed}.
\end{proof}

\begin{recollection}
	Let $\cat C$ be a big tt-$\infty$-category and consider $A \in \CAlg(\cat C)$. As discussed in~\cite[Section 4.5]{HA} there is a big tt-$\infty$-category $\mod{A}(\cat C)$ of $A$-module objects internal to $\C$. Extension-of-scalars (or base change) provides a tt-$\infty$-functor $F_A\colon \cat C \to \mod{A}(\cat C)$ (see \cite[Theorems 4.5.2.1 and 4.5.3.1]{HA}) which admits a right adjoint ``forgetful functor'' $U_A \colon \mod{A}(\cat C)\to \cat C$. To ease the notation we will often drop the subscript $A$ and simply write $F$ and $U$ for these functors. The forgetful functor is conservative (i.e., reflects equivalences) and preserves all colimits \cite[Corollary 4.2.3.5]{HA}.  These two facts have a number of consequences as follows:
\end{recollection}

\begin{lemma}\label{lem-proj-formula-holds}
	Let $\cat C$ be a big tt-$\infty$-category and let $A \in \CAlg(\cat C)$.
	\begin{enumerate}
		\item The base change functor $F_A\colon \cat C \to \mod{A}(\cat C)$ preserves the subcategories of dualizable objects and compact objects.
		\item If $\cat C$ is generated by dualizable objects (that is, $\cat C=\loc{\cat C^d}$) then the projection formula holds for the base change functor $F_A$. 
		\item If $\cat C$ is compactly generated, say by $\cat G \subseteq \cat C^c$, then $\mod{A}(\cat C)$ is compactly generated by $F_A(\cat G)$. Moreover, $\mod{A}(\cat C)^c = \mathrm{thick}(F_A(\cat C^c))$.
		\item If $\C$ is rigidly-compactly generated, then $\mod{A}(\C)$ is also rigidly-compactly generated.
	\end{enumerate}
\end{lemma}

\begin{proof}
	This follows from \cref{lem:preserve-gen} and \cref{cor:gen-forced}.
\end{proof}

\begin{remark}\label{rem:weakly-closed-is-closed}
	Let $\cat C$ be a rigidly-compactly generated big tt-$\infty$-category and consider $A \in \CAlg(\cat C)$. It follows from \cref{prop-properties-adjunction} that if $A$ is compact (as an object of~$\cat C$) then $U_A\colon \mod{A}(\cat C) \to \cat C$ preserves compact objects.
\end{remark}

\begin{proposition}\label{prop:basechange}
    Let $F\colon\cat C \to \cat D$ be a colimit-preserving tt-$\infty$-functor between rigidly-compactly generated tt-$\infty$-categories whose right adjoint $U$ is conservative. Then $F$ is extension-of-scalars with respect to $U(\unit_{\cat D})\in \CAlg(\cat C)$: There is an equivalence of tt-$\infty$-categories $\cat D \cong \mod{U(\unit_{\cat D})}(\cat C)$ under which $F\colon \cat C \cat \to \cat D$ is identified with $\cat C \to \mod{U(\unit_{\cat D})}(\cat C)$.
\end{proposition}

\begin{proof}
	Apply \cite[Proposition 5.29]{MNN}. Note that the required conditions are satisfied by \cref{prop-properties-adjunction}.
\end{proof}

We briefly recall the relative version of the base change functor. 

\begin{recollection}
    Let $\cat C$ be a big tt-$\infty$-category and let $f\colon A \to B$ be a morphism in $\CAlg(\C)$. By~\cite[Corollary 4.2.3.2]{HA}, there is an induced restriction functor $U\colon \mod{B}(\cat C)\to \mod{A}(\cat C)$ which is conservative and preserves all small (co)limits by~\cite[Corollary 4.2.3.7(2)]{HA}. This restriction functor admits a  symmetric monoidal left adjoint given by $M\mapsto B\otimes_A M$; see~\cite[Proposition 4.6.2.17 and Remark 4.5.3.2]{HA}.
\end{recollection}

\begin{corollary}\label{cor-double-module-category}
	Let $\cat C$ be a  rigidly-compactly generated tt-$\infty$-category and let $f\colon A \to B$ be a morphism in $\CAlg(\cat C)$. There is an equivalence of tt-$\infty$-categories
	\[	
		\mod{B}(\cat C) \simeq \mod{U(B)}(\mod{A}(\cat C))
	\]
	under which $F_{U(B)}\colon \mod{A}(\cat C)\to \mod{U(B)}(\mod{A}(\cat C))$ identifies with the extension-of-scalars functor $\mod{A}(\cat C)\to \mod{B}(\cat C)$.
\end{corollary}

\begin{proof}
	Apply \cref{prop:basechange} to the base change functor $\mod{A}(\cat C)\to \mod{B}(\cat C)$.
\end{proof}

\section{Descendable algebras}

We start by recalling some definitions and results from~\cite[Section 3]{Mathew2016}, \cite{Mathew2015} and~\cite{MNN}, which are partly inspired by \cite{Balmer2016}. From now on, let $\C$ be a (big) tt-$\infty$-category. 
 
\begin{definition}\label{def:descendable}
	A commutative algebra $A \in \CAlg(\C)$ is \emph{descendable} if the thick ideal generated by $A$ is all of $\C$: $\thickid{A}=\C$. More generally,  we say that a map of commutative algebras $f\colon A \to B$ is descendable if $B$ is descendable as a commutative algebra in $\mod{A}(\C)$.
\end{definition}

\begin{lemma}\label{lem-generates-same-ideal}
	If $A\in \CAlg(\cat C)$ is descendable then
	\[
		\thickid{x} = \thickid{x\otimes A}
	\]
	for all $x \in \cat C$.
\end{lemma}

\begin{proof}
	The hypothesis $\unit \in \thickid{A}$ implies $x \in \thickid{x\otimes A}$  by a standard thick subcategory argument. Hence, $\thickid{x}\subseteq \thickid{x\otimes A}$. The reverse inclusion is immediate.
\end{proof}

\begin{corollary}\label{cor:descendable-is-conservative}
\sloppy	If $A \in \CAlg(\C)$ is descendable then the base change functor $F_A\colon\cat C \to \mod{A}(\cat C)$ is conservative.
\end{corollary}

\begin{remark}\label{rem:algebra-splits}
	For an arbitrary $A \in \CAlg(\cat C)$, consider the base change adjunction
	\[
		F_A\colon \cat C \adjto \mod{A}(\cat C):U_A.
	\]
	Complete the unit $\eta\colon\unit \to A$ to an exact triangle
	\begin{equation}\label{eq:triangle-on-unit}
		W \xrightarrow{\xi} \unit \xrightarrow{\eta} A \to \Sigma W.
	\end{equation}
	Then $F_A(\xi)=0$ in $\mod{A}(\cat C)$ so that $F_A(A) \simeq \Sigma F_A(W) \oplus \unit_A$ in $\mod{A}(\C)$ and $A\otimes A \simeq (\Sigma A\otimes W) \oplus A$ in $\cat C$.
\end{remark}

\begin{lemma}\label{lem:proj-formula-trick}
	Let $\C$ be a rigidly-compactly generated tt-$\infty$-category. For an arbitrary $A \in \CAlg(\C)$, we have
	\[
		\thickid{U_A(x)} = \thickid{U_A(x)\otimes A}
	\]
	for any $x \in \mod{A}(\cat C)$.
\end{lemma}

\begin{proof}
	This follows from \Cref{rem:algebra-splits}. If $x \in \mod{A}(\C)$ is arbitrary, then $x$ is a summand of $x\otimes F_A(A)$. Hence $U_A(x)$ is a summand of $U_A(x \otimes F_A(A)) \simeq U_A(x) \otimes A$, where the last equivalence is the projection formula (see~\cref{lem-proj-formula-holds}). Hence $\thickid{U_A(x)}=\thickid{U_A(x)\otimes A}$.
\end{proof}

\begin{proposition}\label{prop:nil-faithful}
	For $A \in \CAlg(\cat C)$, the following are equivalent:
	\begin{enumerate}
		\item $A$ is descendable;
		\item The base change functor $F_A \colon \cat C \to \mod{A}(\cat C)$ is nil-faithful: if $f$ is a morphism in $\cat C$ such that $F_A(f) = 0$ (i.e., $f \otimes A = 0$), then $f$ is tensor-nilpotent.
	\end{enumerate}
\end{proposition}

\begin{proof}
	This is essentially contained in the work of Balmer \cites{Balmer2016,Balmer2018}; see also \cite[Proposition~3.27]{Mathew2016}. As observed in \cite[Proposition 2.10]{Balmer2018}, 
	\[
		\thickid{A} = \SET{z \in \cat C}{\xi^{\otimes n} \otimes z = 0 \text{ for some } n \ge 1}
	\]
	where $\xi:W\to \unit$ is the morphism in \eqref{eq:triangle-on-unit}. Thus, $A$ is descendable if and only if the morphism $\xi$ is tensor-nilpotent. In particular, note that $\xi\otimes A=0$, that is, $\xi$ is tensor-nilpotent on its cone. Hence (b) implies (a). Conversely, suppose $f \colon x\to y$ is a morphism in $\cat C$ such that $A \otimes f=0$. From 
	\[\begin{tikzcd}
		W \otimes x \ar[d,"W\otimes f"']  \ar[r] & x \ar[dl,dotted]\ar[d,"f"] \ar[r] &A \otimes x \ar[d,"A\otimes f=0"] \ar[r] & \Sigma W \otimes x \ar[d]\\
		W\otimes y \ar[r,"\xi\otimes y"] & y \ar[r] & A\otimes y \ar[r] & \Sigma W \otimes y
	\end{tikzcd}\]
	we see that $f$ factors through $\xi \otimes y$. If $A$ is descendable, then $\xi$ is tensor-nilpotent and this readily implies that $f$ is tensor-nilpotent too.
\end{proof}

\begin{proposition}\label{prop:descendable-surjective}
	Let $\C$ be a rigidly-compactly generated tt-$\infty$-category. For any $A \in \CAlg(\cat C)$, let
	\[
		\varphi\colon\Spc(\mod{A}(\cat C)^c) \to \Spc(\cat C^c)
	\]
	denote the map on Balmer spectra induced by extension-of-scalars.
	\begin{enumerate}
		\item If $A$ is descendable, then $\varphi$ is surjective.
		\item If $A$ is compact and $\varphi$ is surjective, then $A$ is descendable.
	\end{enumerate}
\end{proposition}

\begin{proof}
	For (a): If $A$ is descendable then $F_A\colon \cat C \to \mod{A}(\cat C)$ is conservative by \cite[Proposition 3.19]{Mathew2016}. Hence $\varphi$ is surjective by \cite[Corollary 2.26]{BCHNP2023}; a generalization to arbitrary families of geometric functors will be given in \cite{BCHS-big-surjectivity}.

	For (b): If $A$ is compact then the restriction functor $U_A$ preserves compact objects (\Cref{rem:weakly-closed-is-closed}). Hence the image of $\varphi$ is $\supp(A)$ by \cite[Theorem 1.7]{Balmer2018}. Since $\varphi$ is surjective, $\supp(A)=\Spc(\cat C^c)$. By the classification of (radical) thick ideals of~$\cat C^c$, this implies that $\unit$ is contained in the (radical) thick ideal of $\cat C^c$ generated by $A$. In particular, $\unit \in \thickid{A}$.
\end{proof}

\begin{remark}\label{rem:smashing-warning}
	It is natural to wonder whether part (b) of \Cref{prop:descendable-surjective} holds without the assumption that $A$ is compact; in other words, whether the surjectivity of $\varphi$ characterizes descendable algebras. This is not true in general, as the following two counterexamples demonstrate.
\end{remark}
	
\begin{example}
	Let $R$ be a commutative ring and let $I \subseteq R$ be a nil ideal that is not nilpotent; for example, let $R = \mathbb{C}[x_1,x_2,x_3,\ldots]/(x_1,x_2^2,x_3^3,\ldots)$ and $I = (\overbar{x_1},\overbar{x_2},\ldots)$. Now consider the extension-of-scalars functor $\Der(R) \to \Der(R/I)$ induced by the quotient map $R \to R/I$. On Balmer spectra, this induces a homeomorphism $\spec(R/I) \cong \spec(R)$ but $R/I \in \Der(R)$ is not descendable \cite[Example 11.21]{BhattScholze2017Projectivity}.
\end{example}

\begin{example}
	For another counterexample, let $R$ be the valuation domain which is called $A$ in \cite{Keller94b}. Let $k\coloneqq R/\mathfrak{m}$ be its residue field and let $Q$ denote its field of fractions. As shown by Bazzoni--{\Stovicek} (see \cite[Example 5.24]{BazzoniStovicek17}), $R\to Q\times k$ induces a conservative (non-finite) smashing localization $\Der(R) \to \Der(Q\times k)$ which is surjective on Balmer spectra. However the idempotent algebra $Q\times k$ is not descendable. (Note that an idempotent algebra $f$ is descendable if and only if $f = \unit$.) This example also shows that the conservativity of $A\otimes -$ does not imply that $A$ is descendable, i.e., the converse of \Cref{cor:descendable-is-conservative} does not hold.
\end{example}

\begin{construction}
	Given $A \in \CAlg(\C)$, we can form a cosimplicial object in $\C$
	\[
		\mathrm{CB}^\bullet(A)= \left\lbrace A \rightrightarrows A \otimes A 
		\substack{\rightarrow\\[-1em] \rightarrow \\[-1em] \rightarrow} \ldots\right\rbrace
		\in \Fun(\Delta, \C)
	\]
	called the \emph{cobar construction} on $A$. Then the sequence of partial totalizations
	\[
		\Tot_n( \mathrm{CB}^\bullet(A))= \lim_{[i]\in\Delta, i\leq n}A^{i+1}
	\]
	naturally arranges into a tower, whose inverse limit is given by the totalization $\Tot( \mathrm{CB}^\bullet(A))$. The cobar construction extends to an augmented cosimplicial object
	\[
		\mathrm{CB}^\bullet_{\mathrm{aug}}(A) =
		\left\lbrace\mathbbm{1} \to A \rightrightarrows A \otimes A 
		\substack{\rightarrow\\[-1em] \rightarrow \\[-1em] \rightarrow} \ldots\right\rbrace
		\in \Fun(\Delta_+, \C)
	\]
	where $\Delta_+$ is the augmented simplex category of finite ordered sets.  
\end{construction}

\begin{proposition}\label{prop-descent-tot}
	Suppose that $A \in \CAlg(\C)$ is descendable. Then for all $x\in\C$ and any exact functor $F \colon \C \to \D$ where $\D$ is a stable $\infty$-category, the augmented cosimplicial object 
	\[
		F(x\otimes\mathrm{CB^\bullet_{\mathrm{aug}}}(A))\colon \Delta^+ \to \D
	\]
	is a limit diagram. That is, the natural morphism
	\[
		\eta_x\colon F(x) \to\Tot(F(x \otimes A) \rightrightarrows F(x \otimes A \otimes A)  
		\substack{\rightarrow\\[-1em] \rightarrow \\[-1em] \rightarrow} \cdots )
	\]
	is an equivalence in $\D$. 
\end{proposition}

\begin{proof}
	This is proved in~\cite[Proposition 2.12]{Mathew2018} but we recall the argument for future reference.  We start by noting that if $x=A\otimes y$, then the augmented cosimplicial object $x\otimes \mathrm{CB}^\bullet_\mathrm{aug}(A)$ admits an ``extra degeneracy'' or splitting~\cite[Definition 4.7.2.2]{HA} which is induced by the multiplication map of $A$. It then follows from~\cite[Lemma 6.1.3.16]{HTT} that the functor $F(x\otimes\mathrm{CB^\bullet_{\mathrm{aug}}}(A))$ is a limit diagram. In other words $\eta_x$ is an equivalence if $x=A \otimes y$. Since the class of $x$ for which $\eta_x$ is an equivalence is thick, it follows that $\eta_x$ is an equivalence for all $x\in\thickid{A}$. Since $A$ is descendable we have $\thickid{A}=\C$.
\end{proof}

Recall the following useful characterization of descendable commutative algebras. 

\begin{proposition}[{\cite{Mathew2016}*{Proposition 3.20}}]\label{prop-characterization-descent}
	Consider $A\in\CAlg(\C)$. Then $A$ is descendable if and only if the Tot-tower $\{\Tot_n(\mathrm{CB}^\bullet(A)) \}_{n \geq 0}$ is pro-constant and convergent to $\unit$ (i.e., $\mathrm{CB}^\bullet_{\mathrm{aug}}(A)$ is a limit diagram).
\end{proposition}

This enables us to characterize dualizable descendable algebras internally in $\C^{d}$:

\begin{lemma}\label{lem-descent-dualizable-equivalent}
	The following are equivalent for $A\in\CAlg(\C)$:
	\begin{enumerate}
		\item $A$ is dualizable and descendable in $\C$;
		\item $A$ is dualizable and descendable in $\C^d$;
		\item $A$ is dualizable and $\supp(A) = \Spc(\cat C^d)$.
	\end{enumerate}
\end{lemma}

\begin{proof}
	Recall that $\C^d$ is a symmetric monoidal subcategory of $\C$ so (b) implies~(a). Conversely, suppose that $A$ is dualizable and descendable in $\C$. Then by \Cref{prop-characterization-descent} the Tot-tower $\{\Tot_n(\mathrm{CB}^\bullet(A)) \}_{n \geq 1}\in \mathrm{Tow}(\C)$ is pro-constant and convergent to $\unit$. Since $A$ is dualizable, $\mathrm{CB}^\bullet(A)$ defines a cosimplicial object in $\C^{d}$ and so the resulting Tot-tower belongs to $\mathrm{Tow}(\C^d)\subseteq \mathrm{Tow}(\C)$.  This remains pro-constant and convergent to $\unit\in\C^d$ so $A$ is descendable in $\C^d$ by \Cref{prop-characterization-descent}. For the equivalence between (b) and (c), note that for any dualizable object $A \in \cat C^d$, $\thickid{A}=\cat C^d$ is equivalent to $\unit \in \thickid{A}$. Moreover, since $\cat C^d$ is rigid, the thick ideal generated by $A$ is the same as the radical thick ideal generated by~$A$. So by the classification of radical thick ideals, $\unit \in \thickid{A}$ is equivalent to $\supp(\unit)\subseteq \supp(A)$.
\end{proof}

Next we show how we can use descendability to detect compact objects. 

\begin{proposition}
	Let $\cat C$ be a big tt-$\infty$-category which is generated by dualizable objects. If $A \in \CAlg(\cat C)$ is descendable then $F_A\colon \cat C \to \mod{A}(\cat C)$ reflects compactness: If $x \in \cat C$ is such that $F_A(x)$ is compact in $\mod{A}(\cat C)$ then $x$ is compact in $\cat C$.
\end{proposition}

\begin{proof}
	This is essentially~\cite[Proposition 3.28]{Mathew2016}. We repeat the argument here for completeness. For any set of objects $\{z_i\}$, we need to show that the map 
	\[
		\bigoplus_i \Hom_{\C}(x, z_i) \to \Hom_{\C}(z, \coprod_i z_i)
	\] 
	is an isomorphism. Consider the collection $\CU$ of objects $y$ such that 
	\[
		\bigoplus_i \Hom_{\C}(x, y\otimes z_i) \to \Hom_{\C}(x, y\otimes \coprod_i z_i)
	\]
	is an isomorphism. We would like to show that it contains $\unit$. The collection $\CU$ forms a thick subcategory. Using the $F_A \dashv U_A$ adjunction and the projection formula (\cref{lem:preserve-gen}), one checks that it contains $U_A(s)$ for every $s \in \mod{A}(\cat C)$. Hence it contains $A \otimes \cat C$. That is, it contains $\thick{A\otimes \cat C}=\thickid{A}$, and hence contains~$\unit$.
\end{proof}

We finish this section by recalling the following result on descent theory.

\begin{proposition}[\cite{Mathew2016}*{Proposition 3.22}]\label{prop-category-as-tot}
	Let $\C\in\CAlg(\Pr)$ and consider $A\in\CAlg(\C)$ descendable. Then the base change adjunction $F_A \colon\C\rightleftarrows\mod{A}(\C)\noloc U_A$ is comonadic. In particular, the canonical map 
	\[
		\C \to \Tot(\mod{A}(\C) \rightrightarrows \mod{A\otimes A}(\C)) \substack{\rightarrow\\[-1em] \rightarrow \\[-1em] \rightarrow} \cdots )
	\]
	is an equivalence.
\end{proposition}

\section{Descent for localizing ideals}

The goal of this section is to establish a descent strategy for classifying localizing ideals; see \Cref{prop-equalizer} below. This will be used in \Cref{sec:stratification} to deduce a descent result for stratification.

\begin{definition}
	For a big tt-category $\cat T$, we denote by $\Locid(\cat T)$ the (large) poset of localizing ideals of~$\cat T$ ordered by inclusion.
\end{definition}

\begin{remark}
	Given two big tt-categories $\cat T$ and $\cat S$, and \emph{any} functor $F \colon \cat T \to \cat S$, we define a map 
	\[
		F_* \colon \Locid(\cat T) \to \Locid(\cat S)
	\]
	by sending a localizing ideal $\cat L \subseteq \cat T$ to the localizing ideal generated by $F(\cat L)\subseteq \cat S$, that is $F_*(\cat L)\coloneqq \locid{F(\cat L)}$. If $\cat L=\locid{t_i \mid i \in I}$ is generated by a set of objects, it is tempting to claim that $F_*(\cat L)=\locid{F(\cat L)}$ is equal to $\locid{Ft_i \mid i\in I}$. Note that we always have the containment $F_*(\cat L) \supseteq \locid{Ft_i \mid i\in I}$ but in general this might not be an equality. The next result shows that we do have such an equality under some mild conditions on $F$:
\end{remark}
 
\begin{lemma}\label{lem-inverse-image-ideal}
	Let $F\colon \cat T \to \cat S$ be a functor between big tt-categories. Suppose the preimage $F^{-1}$ preserves localizing ideals. Then
	\[
		\locid{F(\cat E)} = \locid{F(\Locid(\cat E))}
	\]
	for any collection of objects $\cat E \subseteq \cat T$.
\end{lemma}

\begin{proof}
	The $\subseteq$ inclusion is immediate. For the reverse inclusion, observe that $F^{-1}(\locid{F(\cat E)}$ is a localizing ideal by hypothesis. Since it contains $\cat E$, we conclude that it contains $\locid{\cat E}$. Hence $F(\locid{\cat E})$ is contained in $\locid{F(\cat E)}$ and we are done.
\end{proof}

\begin{remark}\label{rem-composition}
	It follows from \Cref{lem-inverse-image-ideal} that if $F\colon\cat T \to \cat S$ and $G\colon\cat S \to \cat R$ are functors and $G^{-1}$ preserves localizing ideals, then $(G\circ F)_* = G_* \circ F_*$.
\end{remark}

\begin{remark}
	The conclusion of \Cref{lem-inverse-image-ideal} says that $F_*(\locid{\cat E})=\locid{F(\cat E)}$ for any collection of objects $\cat E \subseteq \cat T$. In this case, we say that $F_*$ can be ``defined on generators''.
\end{remark}

\begin{example}\label{exa:tt-preserves}
	Let $F\colon \cat T \to \cat S$ be a geometric tt-functor between big tt-categories. Then $F^{-1}$ preserves localizing ideals and so $F_*$ can be ``defined on generators''.
\end{example}

\begin{lemma}\label{lem:conserve-preserves}
	Let $F\colon \cat T\to \cat S$ be a geometric tt-functor between rigidly compactly generated tt-categories. If its right adjoint $U\colon\cat S\to \cat T$ is conservative then $U^{-1}$ preserves localizing ideals. Hence we have
	\[
		U_*(\locid{\cat E}) = \locid{U(\cat E)}
	\]
	for any collection of objects $\cat E\subseteq \cat S$.
\end{lemma}

\begin{proof}
	Recall from \cref{prop-properties-adjunction} that the right adjoint $U$ preserves coproducts. Hence if $\cat L \subseteq \cat T$ is a localizing ideal then $U^{-1}(\cat L)$ is a localizing subcategory of~$\cat S$. The conservativity of $U$ implies (\cref{lem:preserve-gen}(c)) that $\cat S=\loc{F(\cat T^c)}$, hence $\cat S \otimes U^{-1}(\cat L) \subseteq U^{-1}(\cat L)$ if and only if $F(\cat T^c)\otimes U^{-1}(\cat L) \subseteq \cat L$. But this follows from the projection formula (\cref{prop-properties-adjunction}).
\end{proof}

\begin{example}
	Let $A \in \CAlg(\cat C)$ be a commutative algebra in a rigidly-compactly generated tt-$\infty$-category~$\cat C$. \Cref{lem:conserve-preserves} applies to the baschange functor $F_A\colon \cat C \to \mod{A}(\cat C)$ since the right adjoint $U_A$ is conservative.
\end{example}

\begin{lemma}\label{lem:localizing-from-S}
	Let $F\colon \cat T \to \cat S$ be a geometric functor between rigidly-compactly generated tt-categories. Suppose its right adjoint $U$ is conservative. Then:
	\begin{enumerate}
		\item For any $s \in \cat S$, we have $U(s) \in \locid{U(\unitS)}$.
		\item For any $s\in \cat S$, we have $U(s)\in \locid{U(\unitS)\otimes U(s)}$.
		\item For any localizing ideal $\cat L \subseteq \cat T$, we have an equality
			\[
				\locid{U(\unitS)\otimes \cat L} = \locid{U(\cat S)\cap \cat L}.
			\]
	\end{enumerate}
\end{lemma}

\begin{proof}
	For part (a) just observe that $s \in \locid{\unitS}$ hence $U(s) \in U(\locid{\unitS}) \subseteq \locid{U(\unitS)}$ by \Cref{lem:conserve-preserves}. Now consider part (b). Note that $U(\unitS)\otimes U(s) \simeq UFU(s)$ by the projection formula. We complete the counit of the adjunction $(F,U)$ to a triangle
	\[
		FU(s)\xrightarrow{\epsilon } s \xrightarrow{\chi} z\to \Sigma FU(s).
	\]
	Note that $U(\epsilon)$ is a split epimorphism (by the unit-counit equations), hence the  map $U(\chi) \colon U(s)\to U(z)$ is zero. In particular, the triangle splits after applying $U$, so that $U(s)$ is a direct summand of $UFU(s) \simeq U(\unitS)\otimes U(s) $. This proves part (b). Consider part (c). If $U(s) \in \cat L$ then $U(\unitS)\otimes U(s) \in U(\unitS)\otimes \cat L$. Hence, we have an inclusion
	\[
		\locid{U(\cat S)\cap \cat L} \subseteq \locid{U(\unitS) \otimes \cat L}.
	\]
	On the other hand, observe $U(\unit_{\cat S}) \otimes \cat L =UF(\cat L) \subseteq U(\cat S)\cap \cat L$.
\end{proof}

\begin{definition}
	We say that a localizing ideal $\cat L \subseteq \cat T$ is \emph{generated by objects from~$\cat S$} if $\cat L = \locid{U(\cat S) \cap \cat L}$. Bear in mind part (c) of \Cref{lem:localizing-from-S}. 
\end{definition}

\begin{lemma}\label{lem:all-localizing-from-S}
	Assuming the right adjoint $U$ is conservative, the following are equivalent:
	\begin{enumerate}
		\item Every localizing ideal of~$\cat T$ is generated by objects from $\cat S$.
		\item $\cat T=\locid{U(\cat S)}$.
		\item $\unitT \in \locid{U(\unitS)}$.
	\end{enumerate}
\end{lemma}

\begin{proof}
	$(a) \Rightarrow (b)$ is evident.
	$(b) \Rightarrow (c)$ follows from  part (c) of \Cref{lem:localizing-from-S} applied to $\cat L\coloneqq \cat T$.
	$(c) \Rightarrow (a)$ also follows from \Cref{lem:localizing-from-S}: For any $t \in \cat T$, $t \in \locid{U(\unitS)\otimes t}$, hence $\cat L \subseteq \locid{U(\unitS)\otimes \cat L}$, which is then an equality.
\end{proof}

\begin{notation}
	We write $\Locid(F\colon\cat T \uparrow \cat S)$ for the (large) poset of localizing ideals of~$\cat T$ generated from $\cat S$. Note that it contains a largest element $\locid{U(\cat S)}$.
\end{notation}

The following is a generalization of a result of Mathew~\cite[Proposition 6.3]{Mathew2015}.

\begin{proposition}\label{prop:descent-inj}
	Let $F\colon\cat T\to \cat S$ be a coproduct-preserving tt-functor between rigidly-compactly generated tt-categories whose right adjoint $U$ is conservative. The map
	\[
		F_*\colon \Locid(F\colon \cat T\uparrow \cat S) \to \Locid(\cat S)
	\]
	is split injective in the category of posets, where the splitting is induced by the right adjoint $U$. In particular, if $\unitT \in \Locid(U(\unitS))$ then we have a split injection
	\[
		F_*\colon \Locid(\cat T) \to \Locid(\cat S)
	\]
	in the category of posets.
\end{proposition}

\begin{proof}
	Let $\cat L= \locid{U(\unitS)\otimes \cat L}$ be a localizing ideal generated from objects of $\cat S$; see \cref{lem:localizing-from-S}(c). Note that $F_*$ and $U_*$ can be defined on generators by~\cref{exa:tt-preserves} and~\cref{lem:conserve-preserves}. By \cref{rem-composition} it suffices to verify that $(UF)_*(\cat L)=\cat L$. Since $\cat L$ is generated from object of $\cat S$ this translates to verifying that $\locid{(UF)(U(\unitS)\otimes \cat L)}=\locid{U(\unitS)\otimes \cat L}$. By the projection formula, which holds by \cref{prop-properties-adjunction}, we have $(UF)(U(\unitS)\otimes \cat L)=U(\unitS)\otimes U(\unitS) \otimes \cat L$. The required claim follows from the fact that $ \locid{U(\unitS)\otimes U(\unitS)}=\locid{U(\unitS)}$. One containment is clear and the other follows from \cref{lem:localizing-from-S}(c). For the final claim apply \cref{lem:all-localizing-from-S}.
\end{proof}

\begin{example}\label{exa:descent-inj}
	Let $\cat C$ be a rigidly-compactly generated tt-$\infty$-category and let $A\in \CAlg(\cat C)$ be descendable. Then the map
		\[
			F_*\colon \Locid(\cat C) \to \Locid(\mod{A}(\cat C))
		\]
	is split injective in the category of posets.
\end{example}

\begin{remark}
	The studious reader will notice that some of these results only depend on $A \in \CAlg(\cat C)$ having the property that $\unit \in \locid{A}$. We might call such an algebra ``weakly descendable''. Note, if $A$ is compact then it is weakly descendable if and only if it is descendable.
\end{remark}

\begin{remark}
	Our goal is to enhance the split injection of \Cref{exa:descent-inj} to a split equalizer. To this end, we recall some results about base change functors:
\end{remark}

\begin{proposition}
	Let $A\in \CAlg(\C)$ and suppose that the unit map ${f\colon \mathbbm{1} \to A}$ admits a descendable retraction, i.e., a map $r\colon A \to\mathbbm{1}$ in $\CAlg(\C)$ which is descendable and such that $r\circ f=\mathrm{id}$. Then the base change $F\colon\cat C \to \mod{A}(\cat C)$ induces an isomorphism $F_*\colon \Locid(\C)\xrightarrow{\sim}\Locid(\mod{A}(\C))$. 
\end{proposition}

\begin{proof}
	We note that by \Cref{cor-double-module-category}, we have $\C=\mod{\mathbbm{1}}(\C)\simeq \mod{\mathbbm{1}}(\mod{A}(\C))$. Then base change along $r$ and $f$ induce maps
	\[ 
		F_*\colon \Locid(\C) \to \Locid(\mod{A}(\C))  \quad \mathrm{and} \quad  R_*\colon \Locid(\mod{A}(\C)) \to \Locid(\C)
	\] 
	which satisfies $R_* \circ F_*=\mathrm{Id}$ (since $r$ is a retraction for $f$ together with \Cref{rem-composition}). In particular, we see that $R_*$ is surjective. Since $r$ is descendable, \Cref{exa:descent-inj} shows that the map
	\[
		R_*\colon \Locid(\mod{A}(\C)) \to \Locid(\C)
	\]
	is also injective. So $R_*$ is bijective, with inverse $F_*$.
\end{proof}

\begin{example}
	Let $R$ be a commutative ring spectrum and $X$ a finite connected CW complex. Then the map $C^*(X;R)\to R$ given by evaluation at a basepoint is descendable by~\cite{Mathew2016}*{Proposition 3.36}.  This gives a retraction for the canonical algebra map $R \to C^*(X;R)$.  Therefore $\Loc(\mod{R})\simeq \Loc(\mod{C^*(X;R)})$.
\end{example}

\begin{example}
	Let $R$ be a connective commutative ring spectrum which is \mbox{$n$-truncated} so that $\pi_i(R)=0$ for all $i>n$. Suppose that we have a ring map $\pi_0(R)\to R$ which is the identity on $\pi_0$ (for example this always holds if $R$ is a graded commutative ring). Then the canonical truncation map $\tau_{\leq 0}\colon R\to \pi_0(R)$ is descendable by~\cite{Mathew2016}*{Proposition 3.34} and this is a retraction for $\pi_0(R)\to R$. It follows that $\Loc(\mod{\pi_0(R)})\simeq \Loc(\mod{R})$. If $\pi_0(R)$ is noetherian, then we can classify the localizing subcategories of $\mod{\pi_0(R)}$ and hence that of $\mod{R}$ even though $\pi_*(R)$ might not be noetherian. 
\end{example}

\begin{notation}\label{nota-base-change}
	Consider $A \in \CAlg(\C)$ and let $f \colon \unit \to A$ denote the unit map. Then we have ring maps $g\coloneqq f \otimes 1_A \colon \unit \otimes A \to A \otimes A$ and $h\coloneqq 1_A \otimes f \colon A \otimes \unit \to A \otimes A$. Extending scalars along these rings maps provides functors
	\[
	\begin{tikzcd}\
		\C \arrow[r, "F"] & \mod{A}(\C) \arrow[r, shift right, "G"'] \arrow[r, shift left,"H"] 
		& \mod{A\otimes A}(\C).
	\end{tikzcd}
	\]
	Restricting scalars along $f\colon \unit \to A$ and $g\colon A \to A \otimes A$ define functors $U$ and $V$ as depicted below
	\[
	\begin{tikzcd}\
		\C \arrow[r, "F"] & \mod{A}(\C) \arrow[l, bend left=45,"U"]\arrow[r, "G"] 
		& \mod{A\otimes A}(\C) \arrow[l, bend left, "V"].
	\end{tikzcd}
	\]
\end{notation}

\begin{remark}\label{rem:VHFU}
	Note that for all $A$-modules $M\in \mod{A}(\cat C)$, we have a natural equivalence $(V\circ H)(M)\simeq (F\circ U)(M)$. This is because $(V\circ H)(M)=(A \otimes A)\otimes_A M$ with $A$-module structure coming from the right factor of $A\otimes A$. The left factor of $A\otimes A$ is used for calculating the tensor with $M$.
\end{remark}

\begin{remark}\label{rem:double-is-descendable}
	The algebra morphisms $g\colon A \to A\otimes A$ and $h\colon A \to A \otimes A$ are both descendable. Indeed if $\mu \colon A \otimes A \to A$ is the multiplication map, then $g\mu=1_A=f\mu$ so $A$ is a retract of $A \otimes A$ in $\mod{A}(\C)$.
\end{remark}

\begin{proposition}\label{prop-equalizer}
	Let $\cat C$ be a rigidly-compactly generated tt-$\infty$-category and let $A \in \CAlg(\cat C)$. Then the diagram induced by base change
	\[
	\begin{tikzcd}
		\Locid(F\colon \C\uparrow \mod{A}(\cat C)) \arrow[r, "F_*"] & \Locid(\mod{A}(\C)) 
		 \arrow[r, shift right, "G_*"'] \arrow[r, shift left,"H_*"] 
		 & \Locid(\mod{A\otimes A}(\C)).
	\end{tikzcd}
	\]
	is a split equalizer of posets.
\end{proposition}

\begin{proof}
	We will show that the fork is split in the category of posets. In other words we show that there exist $U_*$ and $V_*$ as in the diagram below
	\[
	\begin{tikzcd}
		\Locid(\C\uparrow \mod{A}(\cat C)) \arrow[r, "F_*"] & \Locid(\mod{A}(\C)) 
		\arrow[l, dotted, bend left,"U_*"]
		\arrow[r, shift right, "G_*"'] \arrow[r, shift left,"H_*"] 
		& \Locid(\mod{A\otimes A}(\C)) \arrow[l, dotted, bend left, "V_*"].
	\end{tikzcd}
	\]
	such that $U_* \circ F_*=1$ and $V_* \circ G_*=1$ and $V_* \circ H_*=F_* \circ U_*$. The equations immediately imply that $F_*$ is an equalizer of $H_*$ and $G_*$ (recall \Cref{rem:split-fork-is-equalizer}).
 
	Recall that $g=F(f)\colon A \otimes \unit \to A \otimes A$ is descendable in $\mod{A}(\C)$ (\cref{rem:double-is-descendable}). Note also that $\mod{A\otimes A}(\mod{A}(\C))\simeq \mod{A\otimes A}(\C)$ by~\cref{cor-double-module-category}. Therefore by \Cref{prop:descent-inj} the functors $F_*$ and $G_*$ admit sections $U_*$ and $V_*$ which are induced by restricting of scalars along $f$ and $g$ respectively. Thus $U_* \circ F_*=1$ and $V_* \circ G_*=1$. 
 
	It is only left to show that $V_* \circ H_*=F_* \circ U_*$. Let $\cat L\subseteq \mod{A}(\C)$ be a localizing ideal. Then we find that $(F_* \circ U_* )(\cat L)=\locid{F( U(\cat L))}$ and $(V_* \circ H_*)(\cat L)=\locid{V(H(\cat L))}$ by \Cref{rem-composition}. To conclude we only need to note that there is a natural equivalence $(V\circ H)(M)\simeq (F\circ U)(M)$ for all $A$-module $M$ as discussed in \cref{rem:VHFU}.
\end{proof}

\begin{example}
	If $A$ is descendable then we obtain a split equalizer of posets:
	\[\begin{tikzcd}
		\Locid(\C) \arrow[r, "F_*"] & \Locid(\mod{A}(\C)) 
		 \arrow[r, shift right, "G_*"'] \arrow[r, shift left,"H_*"] 
		 & \Locid(\mod{A\otimes A}(\C)).
	\end{tikzcd}\]
	It is worth making the following relative version explicit.
\end{example}

\begin{corollary}\label{cor-equalizer}
	Let $A \to B$ be a morphism in $\CAlg(\C)$. If $B$ is descendable over~$A$, then the diagram induced by base change
	\[
		\Locid(\mod{A}(\C)) \to \Locid(\mod{B}(\C)) \rightrightarrows
		\Locid(\mod{B\otimes_A B}(\C))
	\]
	is a split equalizer of sets (hence an equalizer of posets). 
\end{corollary}

\begin{proof}
	Using~\cref{cor-double-module-category} we can rewrite the fork as
	\[
		\Locid(\mod{A}(\C)) \to \Locid(\mod{B}(\mod{A}(\C))) \rightrightarrows 
		\Locid(\mod{B\otimes_A B}(\mod{A}(\C))).
	\]
	Now this is a split equalizer of sets by~\cref{prop-equalizer}. Note that descendability of $A \to B$ ensures that all localizing ideals of $\mod{A}(\C)$ are generated from $\mod{B}(\mod{A}(\C))$; see~\cref{lem:all-localizing-from-S}.
\end{proof}

\begin{example}
	Consider a commutative ring spectrum $A$ with $\pi_i(A)=0$ for all $i>0$ and $i \ll 0$, and with $\pi_0(A)=k$ a field. Then the map $A \to k$ is descendable by~\cite{Mathew2018}*{Example 3.5} and so we get an injective map $\Loc(\mod{A}) \to \Loc(\mod{k})= \{0,\mod{k} \}$.  This is clearly surjective and hence bijective. 
\end{example}

\section{Descent for smashing ideals}

In this section we show how one can use a descendable commutative algebra to classify smashing ideals. First we recall a few important definitions and results. 

\begin{definition}
	Let $\cat T$ be a big tt-category. A localization functor $L \colon \cat T \to \cat T$ is said to be \emph{smashing} if $L$ preserves coproducts. By~\cite{HPS}*{Definition 3.3.2} we have $L\simeq L\unit \otimes -$ as functors $\cat T \to \cat T$. A localizing ideal $\cat I$ of $\cat T$ is \emph{smashing} if there exists a smashing localization $L_{\cat I}\colon \cat T \to \cat T$ such that $\ker(L)=\cat I$. We denote by $\Smashid(\cat T)$ the (large) poset of smashing ideals of $\cat T$ ordered by inclusion. For a big tt-$\infty$-category $\C$, we set $\Smashid(\C):=\Smashid(\Ho(\C))$. 
\end{definition}

We recall the following result. 

\begin{theorem}
	Let $\cat T$ be a rigidly-compactly generated tt-category. Then the collection of smashing ideals $\Smashid(\cat T)$ is a frame with meet and join given by $\cat I \wedge \cat I'=\cat I \cap \cat I'$ and $\cat I \vee \cat I'=\loc{\cat I \cup \cat I'}$ respectively.
\end{theorem}

\begin{proof}
	Krause~\cite{Krause1998}*{Theorem 4.9} showed that the collection of smashing subcategories form a set if $\cat T^c$ satisfies a weak form of Brown's representability theorem which the author calls condition (B). This is always satisfied by~\cite{Krause}*{Theorem 5.1.1}. Balmer--Krause--Stevenson~\cite{BKS2020}*{Theorem 1.1} then showed that $\Smashid(\cat T)$ forms a frame with join and meet as given above; see also~\cite{BKS2020}*{Remark 5.12}. 
\end{proof}

\begin{definition}
	Let $\cat T$ be a big tt-category. We say that a map $\lambda\colon \unit \to E$ is an \emph{idempotent object} in $\cat T$ if the induced map $\lambda \otimes E \colon \unit \otimes E \to E \otimes E$ is an isomorphism. We denote by $\mathrm{Idem}(\cat T)$ the (large) poset of isomorphism classes of idempotent objects in $\cat T$ with partial order given by: $[\lambda\colon \unit \to E] \leq [\lambda'\colon \unit \to E']$ if and only if there is a morphism $\lambda \to \lambda'$ in ${\cat T}_{\unit /}$. For a big tt-$\infty$-category $\C$, we set $\mathrm{Idem}(\C):=\mathrm{Idem}(\Ho(\C))$.
\end{definition}
 
\begin{theorem}\label{thm-idem-smash}
	Let $\cat T$ be a big tt-category. Then the posets $\Smashid(\cat T)$ and $\mathrm{Idem}(\cat T)$ are isomorphic via the maps: $\cat I \mapsto L_{\cat I}\unit$ and $A \mapsto \ker(A \otimes - \colon \cat T \to \cat T)$.
\end{theorem}

\begin{proof}
	This is \cite{BF2011}*{Theorem 3.5}. We note that the authors assume their tt-categories are rigidly-compactly generated but this is not used in the proof of the cited result. 
\end{proof}

We now discuss the corresponding notion of idempotent algebras in the world of $\infty$-categories. 

\begin{definition}
	Let $\C$ be a big tt-$\infty$-category. We say that $A\in\CAlg(\C)$ is an \emph{idempotent algebra} if the map $\eta\otimes A \colon A \to A\otimes A$ is an equivalence. This is equivalent to the multiplication map $\mu \colon A \otimes A \to A$ being an equivalence. We let $\CAlg^\idem(\C)$ denote the full subcategory of $\CAlg(\C)$ spanned by the idempotent algebras. We say $A \leq A'$ if and only if there exists a map of commutative algebras $h\colon A \to A'$ if and only if $A'\simeq A \otimes A'$; cf.~\cite{BKS2020}*{Proposition~2.4}.
\end{definition}

\begin{proposition}
	For a big tt-$\infty$-category $\C$ the mapping spaces of $\CAlg^\idem(\C)$ are either empty or contractible. Furthermore we have equivalences of $\infty$-categories 
	\[
		N(\Smashid(\C ))\simeq N(\mathrm{Idem}(\C))\simeq \CAlg^\idem(\C).
	\]
\end{proposition}

\begin{proof}
	Lurie~\cite{HA}*{Proposition 4.8.2.9} shows that there is a fully faithful functor 
	\[
		\CAlg^{\idem}(\C) \to \C_{\unit /}, \quad A \mapsto (\eta \colon \unit \to A)
	\]
	whose essential image is the collection of maps $\lambda\colon \unit \to E$ such that $\lambda \otimes E$ is an equivalence. Let us write $\mathcal{E}$ for this essential image. In the proof of the cited result, Lurie also proves that the mapping spaces of $\CAlg^\idem(\C)$ (and hence also of $\mathcal{E}$) are either empty or contractible. Therefore by~\cite{HTT}*{Proposition 2.3.4.18} these categories are equivalent to $0$-categories which in turn can be identified with posets by~\cite{HTT}*{Example 2.3.4.3}. Now it is easy to see that $\mathcal{E} \simeq N(\mathrm{Idem}(\C))$. For the last equivalence we apply \Cref{thm-idem-smash}.
\end{proof}

\begin{notation}
	We write $\widehat{\Cat}_\infty$ for the $\infty$-category of large $\infty$-categories and $\widehat{\Poset}$ for the $\infty$-category of large posets. This latter $\infty$-category is in fact a $1$-category in the sense of~\cite{HTT}*{Definition 2.3.4.1}.
\end{notation}

\begin{construction}
	Let $G \colon \C \to \D$ be a tt-functor between big tt-$\infty$-categories. Then the assignment $A \mapsto GA$ defines a functor $G_* \colon \CAlg^\idem(\C) \to \CAlg^\idem(\D)$. Note that we also get an ordering-preserving map between the corresponding posets since $A\otimes A'\simeq A'$ implies $G(A)\otimes G(A')\simeq G(A')$ as $G$ is (strong) symmetric monoidal. Functoriality for the poset of smashing ideals is slightly more mysterious. However we can use the isomorphism with the poset of idempotent commutative algebras to deduce that $F_* \colon \Smashid(\C)\to \Smashid(\D)$ sends $\cat I$ to $\ker(G(L_{\cat I}\unit)\otimes -\colon \D \to \D)$. Therefore we have well-defined functors 
	\[
		\Smashid \colon \CAlg(\Pr)\to\widehat{\Poset} \quad \mathrm{and} \quad \CAlg^\idem\colon \CAlg(\Pr) \to \widehat{\Cat_\infty}
	\]
	and a natural isomorphism $\xi \colon N(\Smashid)\simeq \CAlg^\idem$.
\end{construction}

\begin{corollary}\label{cor-smashing-commutes-limits}
	The functor $\CAlg^\idem\colon \CAlg(\Pr) \to \widehat{\Cat_\infty}$ preserves all limits. Similarly, the functor ${\Smashid} \colon \CAlg(\Pr) \to  \widehat{\Poset}$ preserves all limits.
\end{corollary}

\begin{proof}
	Recall that a commutative algebra object in $\C$ is a functor $\mathrm{Fin}_*\to \C$ satisfying the Segal conditions (which assert that certain maps into a product are equivalences). As limits in functor categories are calculated pointwise we see that the canonical map $\Fun(\mathrm{Fin}_*, \lim_i \C_i)\to \lim_i \Fun(\mathrm{Fin}_*, \C_i)$ is an equivalence. Thus to prove the first claim it suffices to check that a limit of idempotent algebras is again an idempotent algebra. This can be checked pointwise. This holds as limits preserves equivalences and commute with products. The second claim follows from the first one using the equivalence $\xi \colon N\Smashid \simeq \CAlg^\idem$ and the fact that the nerve functor creates limits since it is conservative and commutes with all limits. 
\end{proof}

\begin{remark}
   \Cref{cor-smashing-commutes-limits} can also be deduced from recent work of Aoki~\cite{Aoki2} which shows that the functor $\Smashid(-)$ is a right adjoint and hence preserves all limits.
\end{remark}

As an application we obtain the following result:

\begin{corollary}\label{cor-smashing-eq}
	Let $\C$ be a big tt-$\infty$-category and consider $A\in\CAlg(\C)$ descendable. Then the diagram induced by base change
	\[
		\Smashid(\C) \to \Smashid(\mod{A}(\C)) \rightrightarrows 
		\Smashid(\mod{A\otimes A}(\C)).
	\]
	is an equalizer of posets. 
\end{corollary}

\begin{proof}
	By \Cref{prop-category-as-tot} we know that $\C\simeq \lim_{\bullet \in \Delta} \mod{ A^{\otimes \bullet+1}}(\C)$ and so by the previous result we get that $\Smashid(\C)\simeq \lim_{\bullet \in \Delta}\Smashid( \mod{A^{\otimes \bullet +1}}(\C))$. Note that $\Smashid(\mod{A^{\otimes \bullet +1}}(\C)))$ defines a cosimplicial object in $\widehat{\Poset}$. Applying \cite{Diracpaper}*{Proposition A.1} to the $1$-category $\widehat{\Poset}$ gives
	\[
		\Smashid(\C)\simeq \Tot_1(\Smashid(\mod{A^{\otimes \bullet +1}}(\C)))
	\]
	which proves the claim.
\end{proof}

\section{Descent for thick ideals and the telescope conjecture}

We now turn to thick ideals of compact objects and their relation to our previous results on smashing and localizing ideals. We will show, in particular, that the telescope conjecture descends along a compact descendable commutative algebra.

\begin{notation}
	Let $\cat T$ be a rigidly-compactly generated tt-category. We write $\Thickid(\cat T^c)$ for the poset of thick ideals of the category of compact objects $\cat T^c$.
\end{notation}

\begin{remark}\label{rem:locid-is-loc}
	If $\cat I \subseteq \cat T^c$ is a thick ideal of compact objects, then the localizing subcategory $\loc{\cat I}\subseteq \cat T$ is automatically a localizing ideal of $\cat T$; cf.~\cite[Remark~1.3]{BHS2021} or \cite[Lemma 1.4.6]{HPS}. Thus $\locid{\cat I}=\loc{\cat I}$. In particular, a localizing ideal of~$\cat T$ is compactly generated as a localizing ideal if and only if it is compactly generated as a localizing subcategory.
\end{remark}

\begin{theorem}[Miller--Neeman]\label{thm:miller-neeman}
	Let $\cat T$ be a rigidly-compactly generated tt-category. Then the assignment $\cat I \mapsto \loc{\cat I}$ defines split injective maps
	\[
		\sigma \colon \Thickid(\cat T^c)\to \Locid(\cat T) \qquad \mathrm{and}
		\qquad 
		\sigma \colon \Thickid(\cat T^c)\to \Smashid(\cat T).
	\]
	The splitting sends a smashing or localizing ideal $\cat L$ to $\cat L \cap \cat T^c$.
\end{theorem}

\begin{proof}
	The localizing subcategory $\loc{\cat I}$ generated by a thick ideal of compact objects $\cat I$ is automatically a localizing ideal (\Cref{rem:locid-is-loc}). The theory of finite localizations establishes that it is a smashing ideal; see \cite[\S 3.3]{HPS}. The claim then follows from the assertion that the following composite
	\[\begin{tikzcd}[column sep=small,row sep=tiny]
		\Thickid(\cat T^c) \ar[r] & \Smashid(\cat T) \ar[r,phantom,"\subseteq"] & \Locid(\cat T) \ar[r] &\Thickid(\cat T^c)\\
		\cat I \ar[r,mapsto] & \loc{\cat I} & \cat L \ar[r,mapsto] & \cat L \cap \cat T^c
	\end{tikzcd}\]
	is the identity. In other words, $\loc{\cat I} \cap \cat T^c = \cat I$ as established by Neeman (\Cref{rem:compact-of-compact-loc}). See also \cite[Theorem 3.3.3]{HPS} or \cite[Theorem 8.1]{Greenlees2019}.
\end{proof}

\begin{definition}
	We say that the \emph{telescope conjecture holds for $\cat T$} if the map $\sigma \colon \Thickid(\cat T^c)\to \Smashid(\cat T)$ is bijective. 
\end{definition}

\begin{remark}
	Given a tt-functor $F\colon\cat T \to \cat S$ which preserves compact objects, we obtain an inclusion-preserving map
	\[
		F_*^c \colon \Thickid(\cat T^c) \to \Thickid(\cat S^c)
	\]
	given by $\cat I \mapsto \thickid{F(\cat I)}$.
\end{remark}

\begin{example}
	Let $\cat C$ be a rigidly-compactly generated tt-$\infty$-category and consider $A \in \CAlg(\cat C)$. The extension of scalars functor $ \cat C \to \mod{A}(\cat C)$ preserves compact objects by~\cref{lem-proj-formula-holds}, so we obtain a map 
	\[
		F_*^c\colon \Thickid(\cat C^c) \to \Thickid(\mod{A}(\cat C)^c).
	\]
\end{example}

\begin{theorem}\label{thm-equualizer-thick-ideals}
	Let $\C$ be a rigidly-compactly generated tt-$\infty$-category and let $A\in \CAlg(\cat C)$ be descendable and compact. The diagram
	\[\begin{tikzcd}
		 \Thickid(\C^c) \arrow[r, "F_*^c"] & \Thickid(\mod{A}(\C)^c) 
		 \arrow[r, shift right, "G_*^c"'] \arrow[r, shift left,"H_*^c"] 
		 & \Thickid(\mod{A\otimes A}(\C)^c)
	\end{tikzcd}\]
	induced by base change (\Cref{nota-base-change}) is an equalizer of posets.
\end{theorem}

\begin{proof}
	As the forgetful functor from posets to sets creates limits, it suffices to show that the diagram is an equalizer of sets. Consider the following diagram
	\[
	\begin{tikzcd}
		\Locid(\C) \arrow[r,"F_*"] & \Locid(\mod{A}(\C))
		\arrow[r, "H_*", yshift=1mm]\arrow[r,"G_*"'] & 
		\Locid(\mod{A\otimes A}(\C)) \\
		\Thickid(\C^c) \arrow[u, "\sigma"]
		\arrow[r, "F^c_*"] & 
		\Thickid(\mod{A}(\C)^c) 
		\arrow[u,"\sigma_{A}"] \arrow[r, "H^c_*", yshift=1mm]\arrow[r,"G^c_*"'] & 
		\Thickid(\mod{A\otimes A}(\C)^c).\arrow[u,"\sigma_{A\otimes A}"] 
	\end{tikzcd}
	\]
	where the vertical maps are those of \Cref{thm:miller-neeman} and hence injective. One readily checks that the diagram commutes using \Cref{rem:locid-is-loc} and \Cref{exa:tt-preserves}. The top fork is a split equalizer by \Cref{prop-equalizer} and the hypothesis that $A$ is descendable. Recall from the proof that a section for~$F_*$ is induced by the forgetful functor $U\colon \mod{A}(\C) \to \C$. Therefore we are in the situation of \Cref{lem-mega-lemma} and the bottom fork is an equalizer precisely if for any thick ideal $\cat J \subseteq \mod{A}(\cat C)^c$ satisfying $H_*^c(\cat J)=G_*^c(\cat J)$, the localizing ideal $\locid{U(\cat J)} \subseteq \cat C$ is compactly generated. This is evidently the case if $U$ preserves compact objects (recall \Cref{rem:locid-is-loc}) and this follows from the assumption that $A$ is compact (\cref{cor:weakly-closed-is-closed}).
\end{proof}

\begin{theorem}\label{thm-telescope}
	Let $\C$ be a rigidly-compactly generated tt-$\infty$-category and let $A\in \CAlg(\C)$ be descendable and compact. If the telescope conjecture holds for $\mod{A}(\C)$, then it holds for $\C$ too. 
\end{theorem}

\begin{proof}
	We apply \Cref{cor-smashing-eq} and \Cref{thm-equualizer-thick-ideals} and get a commutative diagram of equalizers
	\[
	\begin{tikzcd}
		\Smashid(\C) \arrow[r] & \Smashid(\mod{A}(\C)) 
		\arrow[r, shift left]\arrow[r,shift right] & 
		\Smashid(\mod{A\otimes A}(\C)) \\
		\Thickid(\C^c) \arrow[u, hook, "\sigma"]\arrow[r] & 
		\Thickid(\mod{A}(\C)^c) 
		\arrow[u,"\sigma","\simeq"'] \arrow[r,shift right]\arrow[r,shift left] & 
		\Thickid(\mod{A\otimes A}(\C)^c).\arrow[u,"\sigma", hook] 
	\end{tikzcd}
	\]
	The vertical arrows are always split injective by \Cref{thm:miller-neeman}. The result now follows from \Cref{lem-diagram-chase}.
\end{proof}

\begin{example}\label{ex:tmf}
	Localized at the prime $2$ there is a ring map $tmf\to tmf_1(3)$. Following Hopkins and Mahowald, Matthew~\cite{Mathewtmf} showed that there is a \mbox{$2$-local} finite complex with torsion-free homology $DA(1)$ and an equivalence of $tmf$-modules $tmf\wedge DA(1)\simeq tmf_1(3)$ so the map $tmf\to tmf_1(3)$ is finite (i.e., $tmf_1(3)$ is a compact $tmf$-module). By the thick subcategory theorem, the thick subcategory that $DA(1)$ generates contains the ($2$-local) sphere so $tmf \to tmf_1(3)$ is descendable. We know that $\pi_*(tmf_1(3))=\Z_{(2)}[v_1,v_2]$ with generators in degree $2$ and $6$ so~\cite{AmbrogioStanley}*{Corollary 1.4} implies that $\mod{tmf_1(3)}$ is stratified by $\pi_*(tmf_1(3))$ and the telescope conjecture holds for $\mod{tmf_1(3)}$. It then follows from \cref{thm-telescope} that the telescope conjecture holds for $\mod{tmf}$ too.
\end{example}

\section{Descent for Balmer spectra}

We recall a few facts about the Balmer spectrum following~\cite{Kock-Pitsch} rather than the standard reference~\cite{Balmer}. 

Recall the notion of a radical thick ideal from \cref{def:thick-subcategories}.

\begin{definition}
	Let $\cat T$ be an essentially small tt-category.
	\begin{enumerate}
		\item Let $\RadThickid(\cat T)$ denote the partially ordered set of radical thick ideals of~$\cat T$ ordered by inclusion. This is a coherent frame by~\cite[Theorem~3.1.9]{Kock-Pitsch} with lattice operations given by $\cat I \wedge \cat I' = \cat I \cap \cat I'$ and $\cat I \vee \cat I' = \sqrt{\cat I \cup \cat I'}$. The finite elements are the principal radical thick ideals, i.e., those of the form~$\sqrt{a}$ for some $a\in\cat T$. We denote by $\RadThickid(\cat T)^f$ the distributive lattice of finite elements.
		\item The \emph{Balmer spectrum} $\Spc(\cat T)$ is the spectral space associated to the distributive lattice $(\RadThickid(\cat T)^f)^{\mathrm{op}}$ via Stone duality; see \Cref{thm-stone-duality}.
	\end{enumerate} 
\end{definition}

\begin{remark}
	Recall that if $\cat T$ is rigid, then any thick ideal is radical. In this case we will simply write $\Thickid(\cat T)$ for the coherent frame of thick ideals, and $\Thickid(\cat T)^f$ for the associated distributive lattice of finite elements.
\end{remark}

\begin{theorem}
	The spectral space $\Spc(\cat T)$ defined above is isomorphic to the spectral space defined in~\cite{Balmer}.
\end{theorem}

\begin{proof}
	By~\cite[Corollary~3.4.2]{Kock-Pitsch} the Balmer spectrum defined in~\cite{Balmer} is the Hochster dual of $\spec (\RadThickid(\cat T)^f)$. But the Hochster dual in the category of distributive lattices is the opposite lattice (\Cref{rem:hochster-duality}).
\end{proof}

\begin{remark}\label{rem:radthickop}
	Note that the distributive lattice $(\RadThickid(\cat T)^f)^{\mathrm{op}}$ is isomorphic to the lattice of quasi-compact open subsets of the Balmer spectrum $\Spc(\cat T)$ via $\sqrt{a} \mapsto \supp(a)^c$.
\end{remark}

Let us give a more concrete description of the meet operation in the distributive lattice $\Thickid(\cat T)^f$:

\begin{lemma}\label{lem:new-meet}
	For any $a,b \in \cat T$, we have $\sqrt{a} \wedge \sqrt{b} = \sqrt{a \otimes b}$.
\end{lemma}

\begin{proof}
	Under Stone duality, the Balmer spectrum of $\cat T$ corresponds to the distributive lattice $(\RadThickid(\cat T)^f)\op$. In particular, by \cref{rem:radthickop}, we have an isomorphism of distributive lattices between $(\RadThickid(\cat T)^f)\op$ and the distributive lattice of quasi-compact open sets, $\Omega(\Spc(\cat T))^f$, under which $\sqrt{a}$ corresponds to $\supp(a)^c$. Thus $\sqrt{a}\wedge \sqrt{b}$ in $\RadThickid(\cat T)^f$ corresponds to 
	\[
		\supp(a)^c \cup \supp(b)^c = (\supp(a) \cap \supp(b))^c = (\supp(a \otimes b))^c.
	\]
	Hence $\sqrt{a} \wedge \sqrt{b} = \sqrt{a\otimes b}$.
\end{proof}

\begin{definition}
	If $Y \subseteq \Spc(\cat T)$ is a Thomason subset, then it corresponds to the radical thick ideal of $\cat T$ given by $\cat T_Y \coloneqq \SET{a \in \cat T}{\supp(a) \subseteq Y}$. We denote by $\RadThickid(\cat T; Y)$ the poset of radical thick ideals of $\cat T$ contained in $\cat T_Y$, ordered by inclusion. Note that this is a subframe of $\RadThickid(\cat T)$.  If $\cat T$ is rigid, we simplify the notation to $\Thickid(\cat T; Y)$ as all thick ideals are radical.
\end{definition}

\begin{lemma}
	If $Y \subseteq \Spc(\cat T)$ is a closed Thomason subset, then $\RadThickid(\cat T; Y)$ is a coherent frame. 
\end{lemma}

\begin{proof}
	We need to establish that the greatest element $1=\cat T_Y$ is finite and that the meet of two finite elements is again finite. By~\cite[Proposition 2.14]{Balmer}, there exists $z\in\cat T$ such that $\supp(z)=Y$ so $\cat T_Y=\sqrt{z}$ is finite. That the meet of two finite elements is again finite follows immediately from \cref{lem:new-meet}.
\end{proof}

\begin{remark}
	From another perspective, a Thomason closed subset $Y \subseteq \Spc(\cat T)$ is a spectral subspace (see \cite[Theorem~2.1.3]{Spectralbook}) and $\ZarTYf$ is the corresponding distributive lattice. The inclusion $Y \subseteq \Spc(\cat T)$ amounts to the map $\ZarTf \to \ZarTYf$ given by intersecting with $\cat T_Y$, in other words $\sqrt{a} \mapsto \sqrt{a} \cap \cat T_Y = \sqrt{a\otimes z}$ where $\supp(z)=Y$.
\end{remark}

\begin{definition}
	A \emph{support theory} on $\cat T$ is a pair $(L,d)$ where $L$ is a distributive lattice and $d\colon \mathrm{obj}(\cat T) \to L$ is a map satisfying:
	\begin{itemize}
		\item[(1)] $d(0)=0$ and $d(1)=1$;
		\item[(2)] $d(\Sigma a)= d(a)$ for all $a\in \cat T$;
		\item[(3)] $d(a \oplus b)=d(a)\vee d(b)$ for all $a,b\in \cat T$;
		\item[(4)] $d(a \otimes b)=d(a) \wedge d(b)$ for all $a,b \in \cat T$;
		\item[(5)] If $a \to b \to c \to \Sigma a$ is an exact triangle in $\cat T$, then $d(b)\leq d(a)\vee d(c)$.
	\end{itemize}
	A morphism of supports from $(L,d)$ to $(L', d')$ is a lattice homomorphism which is compatible with $d$ and~$d'$. 
\end{definition}

\begin{proposition}\label{prop:initial-support}
	Let $\cat T$ be an essentially small tt-category. Define a function as follows
	\[
		d \colon \mathrm{obj}(\cat T) \to \RadThickid(\cat T)^f, \qquad a \mapsto \sqrt{a}.
	\]
	Then the pair $(\RadThickid(\cat T)^f, d)$ is the initial support theory on $\cat T$.
\end{proposition}

\begin{proof}
	See~\cite[Theorem 3.2.3]{Kock-Pitsch}.
\end{proof}

\begin{remark}\label{rem-functoriality}
	Let $F\colon \cat T \to \cat T'$ be a tt-functor between essentially small tt-categories. Then one readily checks that the assignment $a  \mapsto \sqrt{Fa}$ is a support theory for $\cat T$. (The key observation is that the intersection of principal radical thick ideals is given by the tensor-product; see the proof of \cite[Lemma~3.2.2]{Kock-Pitsch}.) Therefore by \Cref{prop:initial-support} we get a homomorphism of distributive lattices 
	\[
		F_*\colon \RadThickid(\cat T)^f \to \RadThickid(\cat T')^f, \quad 
		\sqrt{a} \mapsto \sqrt{Fa}.
	\]
	Under Stone duality, this corresponds to the spectral map $\Spc(F)\colon \Spc(\cat T')\to\Spc(\cat T)$ given by $\cat P \mapsto F^{-1}(\cat P')$.
\end{remark}

\begin{example}\label{ex-ext-scalar-coherent}
	Let $\C$ be a rigidly-compactly generated tt-$\infty$-category and consider $A\in \CAlg(\C)$. Recall that the extension-of-scalars functor $F=A\otimes- \colon \C \to \mod{A}(\C)$ preserves compact objects by \Cref{lem-proj-formula-holds}. Therefore we obtain a map of distributive lattices
	\begin{align*}
		 F^c_* \colon \Thickid(\C^c)^f &\to \Thickid(\mod{A}(\C)^c)^f\\
		 \sqrt{a} &\mapsto \sqrt{Fa}
	\end{align*}
\end{example}

\begin{remark}\label{rem-induced-map-finite-elem}
	If $G\colon \cat T \to \cat T'$ is \emph{any} functor, then the assignment
	\[
		G_* \colon \RadThickid(\cat T) \to \RadThickid(\cat T'), \qquad \cat I\mapsto \sqrt{G(\cat I)}
	\]
	is a well-defined function, but in general it won't be a frame morphism, nor will it preserve finite elements. However, if we assume that $G^{-1}$ preserves radical thick ideals, then we obtain a well-defined map of sets
	\[
		G_* \colon \RadThickid(\cat T)^f \to \RadThickid(\cat T')^f, \qquad \sqrt{a} \mapsto \sqrt{Ga}
	\]
	between distributive lattices; cf.~\Cref{lem-inverse-image-ideal}.
\end{remark}

\begin{example}\label{ex-forgetful-coherent}
	Let $F\colon \cat T\to \cat S$ be a geometric tt-functor between rigidly-compactly generated tt-categories. Suppose that the right adjoint $U$ is conservative and $U(\unitS)$ is compact. Then $U$ preserves all compact objects (by \cref{cor:weakly-closed-is-closed}) and we claim that $U^{-1}$ preserves (radical) thick ideals of compact objects. Hence the right adjoint provides a well-defined map of sets between the distributive lattices
	\[
		U^c_* \colon \Thickid(\cat S^c)^f \to \Thickid(\cat T^c)^f.
	\]
	To prove the claim let $\cat J \subseteq \cat T^c$ be a thick ideal. Then $U^{-1}(\cat J)$ is a thick subcategory, since $U$ is an exact functor. The conservativity of the right adjoint $U$ implies that $\cat S^c = \thick{F(\cat T^c)}$ by \cref{lem:preserve-gen}(c). By the projection formula, $F(\cat T^c) \otimes U^{-1}(\cat J)\subseteq U^{-1}(\cat J)$. Hence $\thick{F(\cat T^c)} \otimes U^{-1}(\cat J)\subseteq U^{-1}(\cat J)$.
\end{example}

\begin{lemma}\label{lem:split-Balmer}
	Let $F\colon\cat T\to \cat S$ be a geometric tt-functor between rigidly-compactly generated tt-categories. Suppose the right adjoint $U$ is conservative and $U(\unitS)$ is compact. Then the image $\im \varphi=\supp(U(\unit))$ of the induced map 
	\[
		\varphi\colon\Spc(\cat S^c)\to\Spc(\cat T^c)
	\]
	is Thomason closed and the map of distributive lattices
	\[
		F^c_* \colon \Thickid(\cat T^c;\im \varphi)^f \to \Thickid(\cat S^c)^f
	\]
	is split injective in the category of sets, with section induced by $U$.
\end{lemma}

\begin{proof}
	First note that the image of the induced map on Balmer spectra coincides with the Thomason closed set $\supp(U(\unit))$ by \cite[Theorem 1.7]{Balmer2018}. By \cref{ex-ext-scalar-coherent} and \cref{ex-forgetful-coherent} we have maps of sets
	\[
		\Thickid(\cat T^c)^f \xrightarrow{F^c_*} \Thickid(\cat S^c)^f \xrightarrow{U^c_*} \Thickid(\cat T^c)^f.
	\]
	The composite maps the radical thick ideal $\sqrt{a} \in \Thickid(\cat T^c)^f$ to the radical thick ideal generated by $U(F(a))={U(\unitS)\otimes a}$. Thus the image of this composite lies in $\Thickid(\cat T^c,\im \varphi)^f$ and the composite
	\[
		\Thickid(\cat T^c;\im \varphi)^f \hookrightarrow \Thickid(\cat T^c)^f \xrightarrow{U^c_*\circ F^c_*} \Thickid(\cat T^c,\im\varphi)^f.
	\]
	is the identity since $\sqrt{a} = \sqrt{U(\unitS) \otimes a}$ when $\supp(a) = \supp(U(\unitS) \otimes a) = \supp(U(\unitS))\cap \supp(a)$, that is, when $\supp(a) \subseteq \im \varphi$.
\end{proof}

\begin{example}\label{ex:split}
	Let $\cat C$ be a rigidly-compactly generated tt-$\infty$-category and let $A \in \CAlg(\cat C)$ be a compact commutative algebra. Then the map of distributive lattices
	\[
		F^c_*\colon \Thickid(\cat C^c;\supp(A))^f \to \Thickid(\mod{A}(\cat C)^c)^f
	\]
	is split injective in the category of sets. 
\end{example}

We are finally ready to prove our descent result for Balmer spectra.

\begin{theorem}\label{cor:coeq-of-spectral}
	Let $\cat C$ be a rigidly-compactly generated tt-$\infty$-category. Let $A\in \CAlg(\cat C)$ be a compact commutative algebra with support $\supp(A) \subseteq \Spc(\cat C^c)$. Then the diagram induced by base change (\Cref{nota-base-change})
	\[\begin{tikzcd}[column sep=large]
		 \Spc(\mod{A\otimes A}(\C)^c )\arrow[r, shift left, "\Spc(H^c_*)"] 
		  \arrow[r, shift right, "\Spc(G^c_*)"']& \Spc(\mod{A}(\C)^c) 
		 \arrow[r, "\Spc(F^c_*)"] 
		 & \supp(A)
	\end{tikzcd}\]
	is a coequalizer of spectral spaces.
\end{theorem}

\begin{proof}
	As before note that any ideal in $\C^c$ is radical by \Cref{rem-all-ideal-radical} and that the solid diagram of distributive lattices
	\[\resizebox{\columnwidth}{!}{$\displaystyle
	\begin{tikzcd}[ampersand replacement=\&]
		\Thickid(\C^c;\supp(A))^f \arrow[r, "F^c_*"] \& \Thickid(\mod{A}(\C)^c)^f 
		\arrow[r, shift right, "G^c_*"'] \arrow[l, bend left, dotted,"U^c_*"]\arrow[r, shift left,"H^c_*"] 
		\& \Thickid(\mod{A\otimes A}(\C)^c)^f\arrow[l, bend left, dotted, "V^c_*"]
	\end{tikzcd}
	$}
	\]
	is well-defined by \Cref{ex-ext-scalar-coherent}. We will show that this fork is split in the category of sets and so in particular is an equalizer. It then follows from \Cref{lem-forgetful-creates-limits} that the above is a limit diagram in $\DLat$. Stone duality will then provide the colimit diagram in the statement of the Corollary.

	As for localizing ideals (see \Cref{prop-equalizer}) we need to construct maps $U^c_*$ and $V^c_*$ as depicted in previous display such that $U^c_* \circ F^c_*=1$, $V^c_* \circ G^c_*=1$ and $V^c_* \circ H^c_*=F^c_* \circ U^c_*$. Recall that $g=F(f)\colon A \otimes \unit \to A \otimes A$ is descendable in $\mod{A}(\C)$ (\cref{rem:double-is-descendable}). Note also that $\mod{A\otimes A}(\mod{A}(\C))\simeq \mod{A\otimes A}(\C)$ by~\cref{cor-double-module-category}. Therefore by \Cref{ex:split} the functors $F^c_*$ and $G^c_*$ admit sections $U^c_*$ and $V^c_*$ which are induced by restricting of scalars along $f$ and $g$ respectively. Thus $U^c_* \circ F^c_*=1$ and $V^c_* \circ G^c_*=1$. 

	It is only left to show that $V^c_* \circ H^c_*=F^c_* \circ U^c_*$. This reduces to checking that $\thickid{VH(M)}=\thickid{FU(M)}$ for all compact $A$-modules $M$. In fact there is an equivalence $(V\circ H)(M)\simeq (F\circ U)(M)$ for all $A$-module $M$ as discussed in \Cref{nota-base-change}. The result then follows.
\end{proof}

Under additional assumptions on the commutative algebra, one can show that the spectral coequalizer is in fact a coequalizer in the category of topological spaces:

\begin{theorem}[Balmer]\label{cor-sep-top-coequ}
	Let $\C$ be a rigidly-compactly generated tt-$\infty$-category. Let $A\in\CAlg(\C)$ be descendable, compact and separable of finite tt-degree. Then the diagram induced by base change
	\[
	\begin{tikzcd}[column sep=large]
		\Spc(\mod{A\otimes A}(\C)^c) \arrow[r, shift left, "\Spc(H^c_*)"] 
		\arrow[r, shift right, "\Spc(G^c_*)"']& \Spc(\mod{A}(\C)^c) 
		\arrow[r, "\Spc(F^c_*)"] 
		& \Spc(\C^c)
	\end{tikzcd}
	\]
	is a coequalizer of topological spaces.
\end{theorem}

\begin{proof}
	Note that $\supp(A)=\Spc(\C^c)$ as $A$ is descendable. Then the result follows from ~\cite{Balmer2016}*{Theorem 3.14}.
\end{proof}

\begin{remark}\label{comparison-with-Balmer}
	Strictly speaking, in \cite{Balmer2016} Balmer works entirely at the level of tensor triangulated homotopy categories and considers categories of modules in those homotopy categories. (The main result of \cite{Balmer2011} shows that the category of modules over a separable ring object in a tensor triangulated category admits a unique structure of triangulated category which is compatible with the original one.) If $A \in \CAlg(\cat C)$ is a highly structured separable commutative algebra, we can invoke the results of \cite{DellAmbrogioSanders18} (see \cite[Corollary 1.11]{DellAmbrogioSanders18} and \cite[Proposition 3.8]{Sanders21pp}) to conclude that $\Ho(\mod{A}(\cat C)) \cong \mod{\Ho A}(\Ho\cat C)$ where $\Ho A \in \CAlg(\Ho \cat C)$ is the induced commutative separable algebra in the homotopy category. The coequalizer in \Cref{cor-sep-top-coequ} can then be identified with the coequalizer in \cite{Balmer2016}*{Theorem~3.6}. On the other hand, Balmer's theorem holds without assuming the tensor triangulated categories have an underlying model.
\end{remark}

\begin{remark}\label{rem:balmer-comparison}
    One can contemplate an alternative approach to \cref{cor-sep-top-coequ} by starting from \cref{cor:coeq-of-spectral} and then exhibiting conditions on $A$ which guarantee that the spectral and topological coequalizers agree: If $R$ is the equivalence relation on $\Spc(\mod{A}(\cat C)^c)$ given by the coequalizer, then the hypotheses of \cref{cor:comparison_coequalizers} are satisfied. Balmer \cite{Balmer2016}*{Lemma 3.8} shows that $\varphi^{-1}(\varphi(\cat P))\subseteq R^{-1}(\overline{\{\cat P\}})$ and the fact that $\varphi^{-1}(\varphi(\cat P))\subseteq \Spc(\mod{A}(\cat C)^c)$ is a $T_1$-space is proved in \cite{Balmer2016}*{Lemma~3.10(a)}. Thus the topological and spectral coequalizers are homeomorphic.
\end{remark}

\section{Descent for stratification}\label{sec:stratification}

In this section we will work with a rigidly-compactly generated tt-category $\cat T$ whose Balmer spectrum of compact objects $\Spc(\cat T^c)$ is weakly noetherian (see \cite[Definition 2.3]{BHS2021}). We will follow \cites{BHS2021,BCHS-costratification} for notation and terminology. Under these assumptions there is a well-defined notion of support for objects of $\cat T$,
	\[
		\Supp(t) \coloneqq \SET{\cat P \in \Spc(\cat T^c)}{g_{\cat P} \otimes t \neq 0} \subseteq \Spc(\cat T^c),
	\]
which lies in the Balmer spectrum of compact objects and which is defined using certain $\otimes$-idempotent objects $g_{\cat P} \in \cat T$; see \cite[Section 2]{BHS2021}. This notion of support extends to a map
	\begin{equation}\label{eq-supp}
		\Supp \colon \Locid(\cat T) \to \CP(\Spc(\cat T^c)), \qquad \cat L \mapsto \Supp(\cat L)=\bigcup_{X\in \cat L}\Supp(X)
	\end{equation}  
where $\CP$ denotes the (contravariant) power set functor. This map is split surjective with section given by 
	\[
		\cat L_{(-)} \colon \CP(\Spc(\cat T^c)) \to \Locid(\cat T), \qquad Y \mapsto \cat L_Y=\{t\in \cat T \mid \Supp(t) \subseteq Y\};
	\]
see~\cite[Lemma 3.4]{BHS2021}. We recall the following result. 
 
\begin{theorem}[\cite{BHS2021}*{Theorem A}]
	Let $\cat T$ be a rigidly-compactly generated tt-category with $\Spc(\cat T^c)$ weakly noetherian. Then the following are equivalent:
	\begin{itemize}
		\item The map in \eqref{eq-supp} is bijective.
		\item The local-to-global-principle holds for $\cat T$ and for each $\cat P \in \Spc(\cat T^c)$, the localizing ideal $\locid{g_{\cat P}}$ is a minimal localizing ideal of $\cat T$.
	\end{itemize}
	If these conditions hold we say that $\cat T$ is stratified.
\end{theorem}

\begin{remark}
	Our goal is to study the extent to which stratification descends along a base change functor. Such ``descent results'' for stratification have already been studied in the literature, e.g., in \cite[\S\S 15--17]{BCHS-costratification} and we begin by recalling some results from that work. Let $F\colon \cat T\to \cat S$ denote a geometric functor between rigidly-compactly generated tt-categories and write
	\[
		\varphi\colon\Spc(\cat S^c) \to \Spc(\cat T^c)
	\]
	for the induced map on Balmer spectra. Recall that $F$ has a right adjoint $U$ which in turn has a right adjoint $V$.
\end{remark}

\begin{proposition}[\cite{BCHS-costratification}]\label{prop:f*conservative}
	Let $F\colon\cat T\to \cat S$ be a geometric functor between rigidly-compactly generated tt-categories with weakly noetherian spectra. Consider the following conditions:
	\begin{enumerate}
		\item $\unit_{\cat T} \in \locid{U(\unit_{\cat S})}$;
		\item $V$ is conservative;
		\item $F$ is conservative;
		\item $\varphi$ is surjective.
	\end{enumerate}
	Then $(a) \Rightarrow (b) \Rightarrow (c) \Rightarrow (d)$. If $\cat T$ is stratified or if $U(\unit_{\cat S})$ is compact then all four conditions are equivalent.
\end{proposition}

\begin{proof}
	Implication $(a)\Rightarrow (b)$ is \cite[Remark 13.24]{BCHS-costratification}, $(b)\Rightarrow (c)$ is \cite[Proposition 13.21]{BCHS-costratification}, and $(c)\Rightarrow (d)$ is \cite[Corollary 13.19]{BCHS-costratification}. Finally, ${(d)\Rightarrow (a)}$ if $\cat T$ is stratified or if $U(\unitS)$ is compact by \cite[Corollary 14.24]{BCHS-costratification} and \cite[Proposition 13.33]{BCHS-costratification}, respectively.
\end{proof}

\begin{remark}
	If $F\colon \cat T\to \cat S$ is not conservative then intuitively there is no reason to expect us to be able to descend stratification from $\cat S$ to $\cat T$. Thus, for there to be any hope for us to conclude that $\cat T$ is stratified, intuitively we expect that the equivalent conditions of \Cref{prop:f*conservative} must hold. From this point of view, ``descendability'' (or at least ``weak descendability'') is a natural hypothesis.
\end{remark}

\begin{proposition}[\cite{BCHS-costratification}]\label{prop:LGP-descend}
	Let $F\colon \cat T\to \cat S$ be a geometric functor between rigidly-compactly generated tt-categories with weakly noetherian spectra. Suppose $\unit_{\cat T} \in \locid{U(\unit_{\cat S})}$. If the local-to-global principle holds for $\cat S$, then it holds for $\cat T$, too.
\end{proposition}

\begin{proof}
	This is \cite[Proposition 15.1]{BCHS-costratification}; cf.~\cite[Remark 15.4]{BCHS-costratification}.
\end{proof}

\begin{proposition}
	Let $F\colon \cat T\to \cat S$ be a geometric functor between rigidly-compactly generated tt-categories with weakly noetherian spectra. Assume $\cat S$ is stratified. The following are equivalent:
	\begin{enumerate}
		\item $\cat T$ is stratified and $\varphi$ is surjective.
		\item $\unit_{\cat T}\in  \locid{U(\unit_{\cat S})}$ and the identity
				\[
					\Supp(F(t)) = \varphi^{-1}(\Supp(t))
				\]
			holds for all $t \in \cat T$.
	\end{enumerate}
\end{proposition}

\begin{proof}
	$(a)\Rightarrow (b)$: This follows from \cref{prop:f*conservative} and \cite[Cor.~14.19]{BCHS-costratification}.
    
	$(b)\Rightarrow (a)$: The hypothesis $\unitT\in\locid{U(\unitS)}$ implies that $\varphi$ is surjective by \cref{prop:f*conservative} and that the local-to-global principle holds for $\cat T$ by \cref{prop:LGP-descend}. Let $\cat P \in \Spc(\cat T^c)$ and suppose $0 \neq t \in \locid{g_{\cat P}}$. Then $\Supp(t)=\{\cat P\}$ and hence $\Supp(F(t))=\varphi^{-1}(\Supp(t))=\varphi^{-1}(\{\cat P\})=\Supp(F(g_{\cat P}))$ by hypothesis and \cite[Remark 13.7]{BCHS-costratification}. Since $\cat S$ is stratified, this implies $F(g_{\cat P}) \in \locid{F(t)}$ which implies $UF (g_{\cat P}) \in \locid{t}$ by \cite[(13.4)]{BCHS-costratification}. This establishes $g_{\cat P} \in \locid{t}$ since $\unit_{\cat T}\in  \locid{U(\unit_{\cat S})}$ implies $g_{\cat P} \in \locid{UF(g_{\cat P})}$.
\end{proof}

\begin{corollary}\label{cor:general-strat-descent}
	Let $F\colon \cat T\to \cat S$ be a geometric functor between rigidly-compactly generated tt-categories with weakly noetherian spectra. Suppose that $\cat S$ is stratified and $\unit_{\cat T}\in  \locid{U(\unit_{\cat S})}$. The following are equivalent:
	\begin{enumerate}
		\item $\cat T$ is stratified.
		\item The identity $\Supp(F(t)) = \varphi^{-1}(\Supp(t))$ holds for all $t \in \cat T$.
	\end{enumerate}
\end{corollary}

\begin{remark}
	Our goal is to provide a third equivalent condition for stratification to descend in the case when $F$ is base change along a descendable algebra; see \Cref{thm-stratification-descent} below. We need some preparatory lemmas.
\end{remark}

\begin{lemma}\label{lem:commutativity}
	Let $F\colon \cat T\to \cat S$ be a geometric functor between rigidly-compactly generated tt-categories with weakly noetherian spectra. Suppose that the local-to-global principle holds for $\cat S$. Then the following diagram is commutative:
	\[
	\begin{tikzcd}
		\Locid(\cat T) \arrow[r,"F_*"] & \Locid(\cat S) \\
		 \CP(\Spc(\cat T^c)) \arrow[u, "\cat L_{(-)}"] \arrow[r, "\CP(\varphi)"]& 
		 \CP(\Spc(\cat S^c)). \arrow[u, "\cat L_{(-)}"]
	\end{tikzcd}
	\]
\end{lemma}

\begin{proof}
	Let $Y \subseteq \Spc(\C^c)$ be a subset so $\CP(\varphi)(Y)=\varphi^{-1}(Y)$. Going around the top we obtain $F_*(\cat L_Y)=\locid{F(t) \mid \Supp(t) \subseteq Y}$ while going around the bottom we obtain $\cat L_{\varphi^{-1}(Y)}=\SET{ s \in \cat S }{\Supp(s) \subseteq \varphi^{-1}(Y)}$. The former is contained in the latter by \cite[Remark~14.13]{BCHS-costratification}. Conversely, as explained in the proof of \cite[Proposition 15.1]{BCHS-costratification}, the local-to-global principle for $\cat S$ implies that
	\[
		\locid{F(g_{\cat P})} = \locid{g_{\cat Q} \mid \cat Q \in \varphi^{-1}(\{\cat P\})}
	\]
	for any $\cat P \in \Spc(\cat T^c)$. Thus, if $s \in \cat S$ is such that $\Supp(s) \subseteq \varphi^{-1}(Y)$, then $s \in \locid{s \otimes g_{\cat Q} \mid \cat Q \in \varphi^{-1}(Y)} \subseteq \locid{g_{\cat Q} \mid \cat Q \in \varphi^{-1}(Y)} = \locid{F(g_{\cat P}} \mid \cat P \in Y)$.
\end{proof}

We also need the following simple observation. 

\begin{lemma}\label{lem-power-set-creates-lim}
	The contravariant power set functor $\CP \colon \Set^{\mathrm{op}}\to \Set$ preserves and  detects limits.  
\end{lemma}

\begin{proof}
	Firstly, we claim that $\CP$ is conservative. Let $f\colon X \to Y$ be a map of sets such that $\CP(f)\colon \CP(Y) \to \CP(X)$ is bijective. Since $\CP(f)(\emptyset)=\emptyset$ and $\CP(f)$ is injective, we deduce that $\CP(f)(\{y\})\not= \emptyset$ for all $y\in Y$. Thus, $f$ is surjective. For injectivity, suppose that there exists $x_0,x_1\in X$ such that $f(x_0)=f(x_1)$. Then $\CP(f)^{-1}(\{x_0\})=\{f(x_0)\}=\{f(x_1)\}=\CP(f)^{-1}(\{x_1\})$ and so $x_0=x_1$ by injectivity of $\CP(f)^{-1}$. Therefore $f$ is injective and so bijective.  We also note that the functor $\CP$ preserves all limits as we have a natural equivalence $\CP(-)\simeq\Set(-,\{0,1\})$. Because~$\CP$ is conservative, it follows that it also detects limits. 
\end{proof}

\begin{theorem}\label{thm-stratification-descent}
	Let $\cat C$ be a rigidly-compactly generated tt-$\infty$-category and let $A \in \CAlg(\cat C)$ be descendable. Let $F_A\colon\cat C \to \mod{A}(\cat C)$ denote the base change functor. Suppose that the following hold:
	\begin{enumerate}
		\item[(i)] The Balmer spectra of $\cat C^c$ and $\mod{A\otimes A}(\cat C)^c$ are weakly noetherian.
		\item[(ii)] $\mod{A \otimes A}(\cat C)$ satisfies the local-to-global principle.
		\item[(iii)] $\mod{A}(\cat C)$ satisfies the minimality condition.
	\end{enumerate}
	Then $\mod{A}(\cat C)$ is stratified and the following are equivalent:
	\begin{enumerate}
		\item[(a)] $\cat C$ is stratified.
		\item[(b)] The identity $\Supp(F_A(t))=\varphi^{-1}(\Supp(t))$ holds for all $t \in \cat C$.
		\item[(c)] The diagram induced by the base changed functors 
			\begin{equation}\label{eq:strat-coequalizer}
				\Spc(\mod{A\otimes A}(\C)^c)\rightrightarrows \Spc(\mod{A}(\C)^c) \to \Spc(\C^c)
			\end{equation}
			is a coequalizer of sets.
	\end{enumerate}
\end{theorem}

\begin{proof}
	Extension of scalars along the multiplication map $m\colon A \otimes A \to A$ gives a functor $M\colon \mod{A \otimes A}(\C) \to \mod{A}(\C)$. This  in turn induces a map $\Spc(M_*)$ between Balmer spectra which by functoriality splits both $\Spc(H_*)$ and $\Spc(G_*)$ (in the notation of \Cref{nota-base-change}). It then follows that $\Spc(\mod{A}(\C)^c)$ is a retract of the weakly noetherian spectral space $\Spc(\mod{A\otimes A}(\C)^c)$, so is itself weakly noetherian by~\cite[Remark 2.6]{BHS2021}. Therefore all Balmer spectra considered in the theorem are weakly noetherian. 

	Recall that $F_A(A)=A \otimes A$ is descendable in $\mod{A}(\C)$ by~\cref{rem:double-is-descendable} and that $\mod{A\otimes A}(\mod{A}(\C))\simeq \mod{A\otimes A}(\C)$ by~\cref{cor-double-module-category}. Applying \Cref{prop:LGP-descend} to the descendable algebra $A\otimes A \in\CAlg(\mod{A}(\C))$ shows that the local-to-global principle holds for $\mod{A}(\C)$ since we are assuming that $\mod{A \otimes A}(\cat C)$ satisfies the local-to-global principle. Hence $\mod{A}(\C)$ is stratified because $\mod{A}(\cat C)$ satisfies the minimality condition by hypothesis. The equivalence of (a) and (b) then follows from \Cref{cor:general-strat-descent}, although an alternative proof will also be obtained below. 
 
	Now, we apply \Cref{lem-diagram-chase} to the diagram
	\[
	\begin{tikzcd}
		\Locid(\C) \arrow[r, "F_*"]\arrow[d,shift left, "\Supp"] & \Locid(\mod{A}(\C)) 
		\arrow[d,shift left, "\Supp"] \arrow[r, shift left, "G_*"] 
		\arrow[r, shift right, "H_*"'] & \Locid(\mod{A\otimes A}(\C)) 
		\arrow[d,shift left,"\Supp"]\\
		\CP(\Spc(\C^c)) \arrow[u,shift left,"\cat L_{(-)}"] \arrow[r,"\CP(\varphi)"] & \CP(\Spc(\mod{A}^c)) \arrow[u,shift left, "\cat L_{(-)}"]\arrow[r, shift left, "\CP(\psi)"] 
		\arrow[r, shift right, "\CP(\xi)"'] & \CP(\Spc(\mod{A\otimes A}^c)) \arrow[u,shift left, "\cat L_{(-)}"]
	\end{tikzcd}
	\]
	where the maps are induced by the extension of scalars functors. Condition (i) holds by \Cref{prop-equalizer}. Condition (ii) holds as discussed at the beginning of this section. Condition (iii) holds since $\mod{A}(\C)$ is stratified. Finally, condition (iv) holds by applying \Cref{lem:commutativity} to the two horizontal fork diagrams, using that both $\mod{A}(\cat C)$ and $\mod{A\otimes A}(\cat C)$ satisfy the local-to-global principle. Consider a fourth statement:
	\begin{enumerate}
		 \item[(d)] The diagram  
			\[
				\begin{tikzcd}
				\CP(\Spc(\C^c))\arrow[r,"\CP(\varphi)"] & \CP(\Spc(\mod{A}(\cat C)^c)) \arrow[r, shift left, "\CP(\psi)"] 
				\arrow[r, shift right, "\CP(\xi)"'] & \CP(\Spc(\mod{A\otimes A}(\cat C)^c)) 
				\end{tikzcd}
			\]
			is an equalizer of sets.
	\end{enumerate}
	Then \Cref{lem-diagram-chase} tells us that (a), (c) and (d) are all equivalent. The fact that (c) and (d) are equivalent then follows from \Cref{lem-power-set-creates-lim}.
\end{proof}

\begin{corollary}\label{cor-descent-strat-noetherian}
	Let $\cat C$ be a rigidly-compactly generated tt-$\infty$-category and let $A\in \CAlg(\cat C)$ be descendable. Suppose the spectrum of $\mod{A\otimes A}(\cat C)$ is noetherian and that $\mod{A}(\cat C)$ satisfies the minimality condition. Then the three conditions (a)--(c) of \Cref{thm-stratification-descent} are equivalent. Furthermore, $\Spc(\C^c)$ is noetherian. 
\end{corollary}

\begin{proof}
	As in the beginning of the proof \Cref{thm-stratification-descent}, $\mod{A\otimes A}(\cat C)$ is an extension by a descendable algebra in $\mod{A}(\cat C)$. Thus we have surjections
	\[
		\Spc(\mod{A\otimes A}(\cat C)^c) \to \Spc(\mod{A}(\cat C)^c)\to\Spc(\cat C^c).
	\]
	Since the domain space is noetherian, so are the other two spaces. Finally recall that the local-to-global principle holds if the spectrum is noetherian (\cite[Theorem~3.22]{BHS2021}) so we invoke \Cref{thm-stratification-descent}.
\end{proof}

\begin{remark}\label{rem:comp-spec-top}
	If the descendable algebra $A \in \CAlg(\cat C)$ is compact then the diagram~\eqref{eq:strat-coequalizer} is a coequalizer of spectral spaces by \Cref{cor:coeq-of-spectral} but recall that this does not necessarily mean that it is a coequalizer of topological spaces or sets. We have a commutative diagram 
	\begin{equation}\label{comparison-map-descent-diagram}
	\begin{tikzcd}
		\Spc(\mod{A \otimes A}(\C)^c)\ar[r,shift left]\ar[r,shift right]&\Spc(\mod{A}(\cat C)^c) \ar[r,"\varphi"] \ar[rd] & \Spc(\cat C^c) \\
																		& & T^{\Top} \ar[u,"p"']
	\end{tikzcd}
	\end{equation}
	which relates the spectral coequalizer $\Spc(\cat C^c)$ and the topological coequalizer $T^\Top$. Note that the map $\varphi$ is closed and surjective by \cite[Remark 13.26]{BCHS-costratification} and \cref{rem:weakly-closed-is-closed}, and it follows (\cref{rem:comparison map is closed}) that the comparison map $p$ is also closed and surjective. 
\end{remark}

\begin{corollary}
	In the situation of \cref{thm-stratification-descent}, assume that the descendable algebra $A \in \CAlg(\C)$ is compact and that $\varphi\colon \Spc(\mod{A}(\C)^c)\to \Spc(\C^c)$ has discrete fibers. Then \eqref{eq:strat-coequalizer} is a coequalizer of topological spaces.
\end{corollary}

\begin{proof}
	Bearing in mind \cref{prop:descendable-surjective}, stratification descends from $\mod{A}(\cat C)$ to $\cat C$ by quasi-finite descent \cite[Theorem 17.16]{BCHS-costratification}. Hence, \cref{thm-stratification-descent} implies that \eqref{eq:strat-coequalizer} is a coequalizer of sets. It follows that the comparison map $p:T^\Top \to \Spc(\cat C^c)$ from the topological coequalizer to the spectral coequalizer is a continuous bijection, hence a homeomorphism since it is a closed map (\cref{rem:comp-spec-top}).
\end{proof}

\begin{example}
	Suppose we are in the situation of \Cref{thm-stratification-descent} but assume in addition that $A$ is compact and separable of finite tt-degree. Then condition (c) in the theorem is satisfied by \Cref{cor-sep-top-coequ} so we can descend stratification. In other words, we can descend stratification along a surjective compact separable extension (cf.~\Cref{prop:f*conservative}). This provides a different perspective on descent results along finite \'{e}tale extensions; see \cite{Sanders21pp} and \cite[Section 6]{BHS2021}.
\end{example}

\section{Faithful Galois extensions}

In this section we show that stratification descends along faithful Galois extensions. We first study faithful $G$-Galois extensions for a finite group $G$ in a general big tt-$\infty$-category $\C$, and then specialize to faithful $G$-Galois extensions of commutative ring spectra. In this second case we allow $G$ to be a compact Lie group. 

Let us start by discussing the finite group case.
\begin{definition}
	Let $\C$ be a big tt-$\infty$-category and let $G$ be a finite group. Consider a commutative algebra object with $G$-action $A \in \Fun(BG, \CAlg(\C))$. The \mbox{$\infty$-category} $\mod{A}(\C)$ inherits an action of $G$ and for any $M \in \mod{A}(\C)$ and $g\in G$, we let $g.M \in \mod{A}(\C)$ denote the image of $M$ under the isomorphism $g \colon \mod{A}(\C)\simeq \mod{A}(\C)$. Informally, $g.M$ has underlying object given by $M$ but with module structure given by 
	\[
		A \otimes M \xrightarrow{g \otimes M} A \otimes M \to M.
	\] 
	The $G$-action on $A$ induces a $G$-action on the lattices of thick, localizing and smashing ideals of $\mod{A}(\C)$. Given a localizing ideal $\cat L \subseteq \mod{A}(\C)$, we have
	\[
		g. \cat L =\locid{g.M \mid M \in \cat L}
	\] 
	and similarly for thick and smashing ideals. We say that $\cat L$ is $G$-\emph{invariant} if $g.\cat L=\cat L$ for all $g\in G$. We write 
	\[
		\Locid^G(\mod{A}(\C)) \quad \Thickid^G(\mod{A}(\C)^c)\quad \Smashid^G(\mod{A}(\C))
	\]
	for the posets of $G$-invariant localizing, thick and smashing ideals respectively.
\end{definition}

\begin{definition}\label{def-faithful-galois}
	Let $\C$ be a big tt-$\infty$-category and let $G$ be a finite group. A \emph{faithful} $G$-\emph{Galois extension} in $\C$ is a commutative algebra object with $G$-action $A\in\Fun(BG,\CAlg(\C))$ such that:
	\begin{itemize}
		\item[(a)] The natural map $\unit \to A^{hG}$ is an equivalence.
		\item[(b)] The natural map $A \otimes A \to \prod_G A$ is an equivalence (informally given by $a\otimes a' \mapsto (g \mapsto a\cdot (g a'))$).
		\item[(c)] $A\otimes -\colon \C \to \C$ is conservative. 
	\end{itemize}
\end{definition}

\begin{proposition}\label{prop-galois-desc+compact}
	Let $A$ be a faithful $G$-Galois extension in a rigidly-compactly generated tt-$\infty$-category $\C$. Then $A$ is compact and descendable in $\C$.
\end{proposition}

\begin{proof}
	It is dualizable (and hence compact) by~\cite{Mathew2016}*{Proposition 6.14} and descendable by~\cite{Mathew2016}*{Theorem 3.38}.
\end{proof}

\begin{proposition}\label{prop-galois}
	Let $A$ be a faithful $G$-Galois extension in a rigidly-compactly generated tt-$\infty$-category $\C$. Then base change induces an isomorphism of posets
	\begin{align*}
		\Locid(\C) &\simeq \Locid^G(\mod{A}(\C)) \\ 
		\Thickid(\C^c) &\simeq \Thickid^G(\mod{A}(\C)^c) \\
		\Smashid(\C) &\simeq \Smashid^G(\mod{A}(\C)).
	\end{align*}
	If the telescope conjecture holds for $\mod{A}(\C)$, then it holds for $\C$ too.
\end{proposition}

\begin{proof}
	This is a consequence of our descent results and the fact that any faithful $G$-Galois extension is descendable and compact. Let us expand the argument for localizing ideals.  Since $A$ is Galois, we have a natural identification $A\otimes A\simeq \prod_G A$. Unwinding the definition we see that the two canonical maps ${\mod{A}(\C) \rightrightarrows \mod{A\otimes A}(\C)}$ correspond to 
	\[
		\phi_0,\phi_1 \colon \mod{A}(\C) \rightrightarrows  \mod{\prod_G A}(\C)
	\]
	given by 
	\begin{equation}\label{eq-phi}
		\phi_0(M)=\prod_G M \quad \mathrm{and}\quad \phi_1(M)=\prod_{g\in G}g.M.
	\end{equation}
	Therefore we have an equalizer 
	\[
		\Locid(\C) \to \Locid(\mod{A}(\C)) \rightrightarrows \Locid(\mod{\prod_G A}(\C))
	\]
	induced by the maps $\phi_0$ and $\phi_1$ and the usual base change functor. This means that $\cat L \in \Locid(\mod{A}(\C))$ is in $\Locid(\C)$ precisely if 
	\[
		\locid{\phi_0(M)\mid M \in \cat L}=\locid{\phi_1(M) \mid M \in \cat L}.
	\]
	This is clearly equivalent to $\cat L=\locid{g.M \mid g\in G \; M \in \cat L}$, namely to $\cat L$ being $G$-invariant. A similar proof gives the claim for thick and smashing ideals. The final claim on the telescope conjecture follows from \Cref{thm-telescope}.
\end{proof}

\begin{proposition}\label{prop-Balmer-galois}
	Let $A$ be a faithful $G$-Galois extension in a rigidly-compactly generated tt-$\infty$-category $\C$. The diagram induced by base change
	\[
		\Spc(\mod{A \otimes A}(\C)^c) \rightrightarrows \Spc(\mod{A}(\C)^c) \to \Spc(\C^c)
	\]
	is a coequalizer in the category of topological spaces. Furthermore, there is a canonical homeomorphism $\Spc(\C^c)\cong \Spc(\mod{A}(\C)^c)/G$.
\end{proposition}

\begin{proof}
	This can be deduced from work of Pauwels on quasi-Galois theory; see~\cite{Bregje}*{Theorem 9.1}. Here we give a different proof using point-set topology. Under the identification $A \otimes A\simeq \prod_G A$,  we can rewrite the above fork as
	\[
		\begin{tikzcd}[column sep=large]
		\coprod_{G}\Spc(\mod{A}(\C)^c) \arrow[r, shift left,"\Spc(\phi_1)"] \arrow[r, shift right,"\Spc(\phi_0)"'] & \Spc(\mod{A}(\C)^c) \arrow[r,"\varphi"] &\Spc(\C^c)
	\end{tikzcd}
	\]
	where $\phi_0$ and $\phi_1$ are as in (\ref{eq-phi}). By \Cref{cor:coeq-of-spectral} this fork is a coequalizer diagram in the category of spectral spaces. We note that the coequalizer in the category of topological spaces is given by $\Spc(\mod{A}(\C)^c)/G $. Therefore we have a commutative diagram 
	\[
		\begin{tikzcd}[column sep=large]
		\coprod_{G}\Spc(\mod{A}(\C)^c) \arrow[r, shift left,"\Spc(\phi_1)"] \arrow[r, shift right,"\Spc(\phi_0)"'] & \Spc(\mod{A}(\C)^c) \arrow[dr,"\varphi_{\mathrm{top}}"']\arrow[r,"\varphi"] &\Spc(\C^c)\\
                                                                                                     & & \Spc(\mod{A}(\C))/G \arrow[u,"p"'].
	\end{tikzcd}
	\]
	It follows from \Cref{prop-Ttop-spectral} that the comparison map $p$ is a homeomorphism provided that $\Spc(\mod{A}(\C)^c)/G $ is a spectral space and the map $\varphi_{\mathrm{top}}$ is spectral. This follows from~\cite{Fargues}*{Example 1.7.2} (see also~\cite{Scholze2017}*{Lemma 2.10}).
\end{proof}

\begin{theorem}\label{cor-descent-strat-fin-group}
	Let $A$ be a faithful $G$-Galois extension in a rigidly-compactly generated tt-$\infty$-category $\C$. Suppose that $\mod{A}(\C)$ is stratified with noetherian Balmer spectrum. Then $\C$ is also stratified with noetherian Balmer spectrum given by $\Spc(\mod{A}(\C)^c)/G$.
\end{theorem}

\begin{proof}
	By assumption, the conditions of \Cref{cor-descent-strat-noetherian} are satisfied (note that $\Spc(\mod{A \otimes A}(\C)^c) \simeq  \coprod_{G}\Spc(\mod{A}(\C)^c)$ is noetherian). Then, \Cref{prop-Balmer-galois} shows that $(c)$ of \Cref{thm-stratification-descent} holds, which is equivalent to $\C$ being stratified. 
\end{proof}

In fact, we can descend a stronger notion of stratification along a faithful Galois extension.

\begin{definition}
	We say that a rigidly-compactly generated tt-$\infty$-category is \emph{cohomologically stratified} if $\C$ is stratified and Balmer's comparison map \cite[Definition 5.1]{Balmer2010Spectra} $\rho \colon \Spc(\C^c) \to \spec^h( \pi_*\Hom_{\C}(\unit,\unit))$ is a homeomorphism. 
\end{definition}

\begin{remark}
	Suppose that $\C$ is a noetherian tt-category in the sense of~\cite{BCHNP2023}*{Definition 2.9}. Then $\C$ is cohomologically stratified if and only if $\C$ is stratified by the action of the graded ring $\mathrm{End}^*_\C(\unit)$ in the sense of Benson, Iyengar, and Krause~\cite{BIK2011}; see~\cite{BCHNP2023}*{Remark 2.20} and \cite{Zou-non-noetherian}. Therefore we always have implications:
	\[
	\begin{tikzcd}[column sep=huge]
		 \text{BIK-stratified} \arrow[r, Leftrightarrow,"+\text{Noether.}"] &  \text{cohomologically stratified} \arrow[r, Rightarrow]& \text{stratified}.\\
	\end{tikzcd}
	\]
\end{remark}

\begin{proposition}\label{prop-descent-coh-strat}
	Let $A$ be a faithful $G$-Galois extension in a rigidly-compactly generated tt-$\infty$-category $\C$. Suppose that $\mod{A}(\C)$ is cohomologically stratified. Then so is $\C$. 
\end{proposition}

\begin{proof}
	For any $X \in \cat C$ we let $\pi_*X \coloneqq \pi_*\Hom_{\cat C}(\unit,X)$. There is a descent spectral sequence with $E_2$-term $H^*(G;\pi_*(A))$ converging to $\pi_*(\unit)$. By \cite[Corollary~4.4]{Mathew2016} the descent spectral sequence collapses at a finite stage with a horizontal vanishing line.

	To see that $\C$ is cohomologically stratified, consider the edge homomorphism $\alpha \colon \pi_*(\unit) \to \pi_*(A)^G$ of the descent spectral sequence, whose kernel is a nilpotent ideal \cite[Theorem 4.5]{Mathew2016}. Because every element in positive filtration of the descent spectral sequence is of $|G|$-torsion, we deduce that for every $y \in \pi_*(A)^G$, there exists $k > 0$ such that $y^k$ is in the image of $\alpha$ (compare \cite[Lemma 4.9]{Mathew2015b}). In other words, the map $\alpha$ is an $\mathcal{N}$-isomorphism, and so induces a homeomorphism $\alpha^* \colon \spec^h((\pi_*(A)^G)) \cong \spec^h(\pi_*(A))/G \to \spec^h(\pi_*(\unit))$ on Zariski spectra, for example by the graded version of \cite[Proposition 3.24]{MathewNaumannNoel2019Derived}.

	By naturality of Balmer's comparison map we obtain the commutative diagram displayed in \cref{figure:comparison}.
\begin{figure}[h!]
  \centering
    \[
	\begin{tikzcd}[ampersand replacement=\&]
		\& {\Spc(\mod{A}(\C)^c)/G} \\
		{\Spc(\mod{A}(\C)^c)} \&\& {\Spc(\C^c)} \\
		\& {\spec^h(\pi_*A)/G} \\
		{\spec^h(\pi_*A)} \&\& {\spec^h(\pi_*(\unit))}
		\arrow["{\rho}"', "\sim", from=2-1, to=4-1]
		\arrow[from=4-1, to=4-3]
		\arrow["{\rho}", from=2-3, to=4-3]
		\arrow[from=2-1, to=1-2]
		\arrow["\sim", from=1-2, to=2-3]
		\arrow[from=4-1, to=3-2]
		\arrow["\alpha^*", "\sim"', from=3-2, to=4-3]
		\arrow["\sim"{pos=0.7}, from=1-2, to=3-2]
		\arrow["\varphi"{pos=0.3}, from=2-1, to=2-3, crossing over]
	\end{tikzcd}
\]
\caption{Compatibility of comparison maps}
\label{figure:comparison}
\end{figure}
	Here we have used that the homeomorphism $\Spc(\mod{A}(\cat C)^c)\to \spec^h(\pi_*(A))$ is compatible with the $G$-action, which follows again from naturality of the comparison map. This diagram shows that $\rho \colon \Spc(\cat C^c) \to \spec^h(\pi_*(\unit))$ is a homeomorphism, as required. 
\end{proof}

\begin{example}\label{exa:ko-g-stratification}
	Let $G$ be a finite group, and let $KU_G$ and $KO_G$ denote the complex and real forms of equivariant $K$-theory, and $\mod{KU_G}(\Sp_G)$ and $\mod{KO_G}(\Sp_G)$ the corresponding category of modules in the category $\Sp_G$ of $G$-spectra. The equivariant complexification map $KO_G \to KU_G$ is a faithful $C_2$-Galois extension in $\Sp_G$ by~\cite{MNN}*{Theorem 9.17}. It is shown in~\cite{BCHNP2023}*{Theorem 8.12} that $\mod{KU_G}(\Sp_G)$ is stratified with noetherian Balmer spectrum given by $\spec(\mathrm{RU}(G))$, the Zariski spectrum of the complex representation ring of $G$. Then \Cref{cor-descent-strat-fin-group} applies to show that $\mod{KO_G}(\Sp_G)$ is also stratified with Balmer spectrum given by $\spec(\mathrm{RU}(G))/C_2\cong \spec(\mathrm{RU}(G)^{C_2})$. Note that 
	\[
		\rho \colon \Spc(\mod{KU_G}(\Sp_G)^c) \to \spec^h(\pi_*^G(KU_G)) \cong \spec^h(RU(G)[\beta^{\pm 1}]) 
	\]
	is a homeomorphism in this case \cite[Remark 2.6 and Lemma 8.11]{BCHNP2023}, so descent for cohomological stratification implies that 
	\[
		\rho \colon \Spc(\mod{KO_G}(\Sp_G)^c) \to \spec^h(\pi_*^G(KO_G))
	\]
	is also a homeomorphism. 
\end{example}

\begin{example}
	Let $E_n$ be Morava E-theory and $\mathbb{G}_n$ the Morava stabilizer group. For any finite subgroup $G \subseteq \mathbb{G}_n$ we have a faithful $G$-Galois extension $E_n^{hG}\to E_n$ by~\cite{HMS2017}*{Proposition 3.6}. Since $\mod{E_n}$ is cohomologically stratified by~\cite{AmbrogioStanley}*{Theorem 1.1}, so is $\mod{E_n^{hG}}$.
\end{example}

\begin{remark}
	In the previous section we discussed faithful $G$-Galois extensions in a general big tt-$\infty$-category. We now specialize to faithful $G$-Galois extensions of commutative ring spectra following~\cite{Rognes2008}. 
\end{remark}

\begin{definition}
	Let $G$ be a topological group having the homotopy type of a finite CW complex (for example, a compact Lie group) and let $ B\in\Fun(BG, \CAlg_A)$. Then $A\to B$ is a \emph{faithful $G$-Galois extension} if the following hold:
	\begin{itemize}
		\item[(a)] The canonical map $A \to B^{hG}$ is an equivalence;
		\item[(b)] The canonical map $h\colon B \otimes_A B \to \Hom(G_+,B)$ adjoint to 
			\[
				B \otimes_A B \otimes G_+ \xrightarrow{B \otimes \mathrm{act}} B \otimes_A B \xrightarrow{\mu} B
			\]
			is an equivalence;
		\item[(c)] $B \otimes_A - \colon \mod{A}\to \mod{A}$ is conservative.
	\end{itemize}
\end{definition}

\begin{remark}\label{rem-descent+compact}
    Let $A\to B$ be a faithful $G$-Galois extension of commutative ring spectra. The same argument as in \Cref{prop-galois-desc+compact} shows that $A \to B$ is finite and descendable.
\end{remark}

\begin{notation}
	Let $G$ be a compact Lie group. Write $G_e$ for the identity component of $G$; it is a normal subgroup with quotient $G/G_e=\pi_0 G$. Note that $\pi_0 G$ is a finite group.
\end{notation}

\begin{lemma}\label{lem-allowable}
	The normal subgroup $G_e \subseteq G$ is allowable in the sense of~\cite{Rognes2008}*{Definition 7.2.1}. 
\end{lemma}

\begin{proof}
	We have to verify three conditions:
	\begin{itemize}
		\item[(a)] The collapse map $G \times_{G_e} EG_e \to G/G_e$ induces a stable equivalence on suspension spectra;
		\item[(b)] the projection map $G \to G/G_e$ admits a section up to homotopy;
		\item[(c)] $G_e$ is stably dualizable.
	\end{itemize}
	For (a) it is enough to note that $G$ decomposes as a disjoint union of orbits of $G_e$ so the collapse map is an equivalence even before passing to spectra. To construct the section in (b), it suffices to pick a point in each connected component of $G$. The resulting map $\pi_0G \to G$ splits the projection up to homotopy. Finally, condition~(c) follows from~\cite{Rognes2008}*{Example 3.4.2}.
\end{proof}

\begin{lemma}\label{lem-split-galois}
	Suppose $f\colon A \to B$ is a faithful $G$-Galois extension of commutative ring spectra. Then $f$ is a composite of faithful Galois extensions
	\[
		A \xrightarrow{\pi_0G} B^{hG_e} \xrightarrow{G_e} B
	\]
	where the label of the arrow indicates the corresponding Galois group. 
\end{lemma}

\begin{proof}
	This follows from Rognes' forward Galois correspondence \cite{Rognes2008}*{Theorem 7.2.3} applied to $K=G_e$, which is an allowable normal subgroup of $G$ by \Cref{lem-allowable}.
\end{proof}

\begin{proposition}
	Let $H$ be a compact Lie group which is connected. If $f\colon A \to B$ is a faithful $H$-Galois extension of commutative ring spectra, then induction along $f$ induces a homeomorphism 
	\[
		\Spc(\mod{A}^c)\xrightarrow{\sim}\Spc(\mod{B}^c).
	\]
\end{proposition}

\begin{proof}
	Recall from \Cref{rem-descent+compact} that $f \colon A\to B$ is finite and descendable. Therefore by \Cref{cor:coeq-of-spectral}, we have a coequalizer diagram of spectral spaces
	\[
	\begin{tikzcd}
		\Spc(\mod{B \otimes_A B}^c) \arrow[r,shift left] \arrow[r, shift right] & \Spc(\mod{B}^c) \arrow[r]&  \Spc(\mod{A}^c)
	\end{tikzcd}
	\]
	where the two parallel arrows are induced by the two maps $B \rightrightarrows B \otimes_A B $ given by the left and right unit, informally sending $b$ to $b\otimes 1$ and $ 1\otimes b$, respectively. These maps are split by the multiplication map $B\otimes_A B\to B$ so we get an induced injective map $\Spc(\mod{B}^c)\to \Spc(\mod{B\otimes_A B}^c)$ splitting $\Spc(\mod{B \otimes_A B}^c) \rightrightarrows \Spc(\mod{B}^c) $. We now claim that $\Spc(\mod{B}^c)\to \Spc(\mod{B\otimes_A B}^c)$ is a surjective map. This is because the multiplication map can be identified with the map 
	\[
		B \otimes_A B \simeq \Hom(H_+, B) \to B
	\]
	induced by $\{e\}\to H$. Since up to homotopy $H$ is a connected finite CW complex, this map is descendable by~\cite{Mathew2016}*{Proposition 3.36}. Combining~\cite{Mathew2016}*{Proposition~3.27} with~\cite{Balmer2018}*{Theorem 1.3}, we conclude that $\Spc(\mod{B}^c)\to \Spc(\mod{B\otimes_A B}^c)$ is surjective. It now follows that $\Spc(\mod{B}^c)\to \Spc(\mod{B\otimes_A B}^c)$ is bijective and the two maps $\Spc(\mod{B \otimes_A B}^c) \rightrightarrows \Spc(\mod{B}^c) $ are inverses to it. In particular, the two parallel arrows agree in the category of sets. Thus, the topological coequalizer of $\Spc(\mod{B \otimes_A B}^c) \rightrightarrows \Spc(\mod{B}^c) $ can be identified with $\Spc(\mod{B}^c)$, which is a spectral space. It follows from \Cref{prop-Ttop-spectral} that the spectral and topological coequalizers must agree, giving us the equivalence $ \Spc(\mod{A}^c)\cong\Spc(\mod{B}^c)$.
\end{proof}

\begin{remark}\label{rem-top}
	The argument of the previous proposition shows that the diagram 
	\[
	\begin{tikzcd}
		\Spc(\mod{B \otimes_A B}^c) \arrow[r,shift left] \arrow[r, shift right] & \Spc(\mod{B}^c) \arrow[r]&  \Spc(\mod{A}^c)
	\end{tikzcd}
	\]
	is a coequalizer in topological spaces.
\end{remark}

\begin{corollary}\label{cor-descent-strat-connected}
	Let $H$ be a compact Lie group which is connected and let $ A \to B$ be a faithful $H$-Galois extension of commutative ring spectra. If $\mod{B}$ is stratified with $\Spc(\mod{B}^c)$ noetherian, then so is $\mod{A}$.
\end{corollary}

\begin{proof}
	Combine \Cref{cor-descent-strat-noetherian} with \Cref{rem-top}.
\end{proof}

\begin{theorem}\label{descent-strat-compact-lie}
	Let $A \to B$ be a faithful $G$-Galois extension of commutative ring spectra where $G$ is a compact Lie group. If $\mod{B}$ is stratified with $\Spc(\mod{B}^c)$ noetherian, then so is $\mod{A}$.
\end{theorem}

\begin{proof}
	By \Cref{lem-split-galois}, we can split the problem into faithful Galois descent for a finite group and for a connected group. The result thus follows from \Cref{cor-descent-strat-fin-group} and \Cref{cor-descent-strat-connected}.
\end{proof}

\bibliography{reference}
\vspace{-2em}
\end{document}